\documentclass[11pt]{article}

\usepackage{amsmath,amsthm}
\usepackage{stmaryrd}

\newcommand{\gw}{{\bf gw}}
\newcommand{\cupall}{\pmb{\pmb{\cup}}}

\newcommand{\yes}{\mbox{\sc  yes}}
\newcommand{\no}{\mbox {\sc no}}
\newcommand{\tw}{{\bf tw}}

\newcommand{\lp}{\bigl(}
\newcommand{\rp}{\bigr)}

\newcommand{\intv}[2]{\left \llbracket #1, #2 \right \rrbracket}
\newcommand{\sig}{{\bf sig}}
\newcommand{\tab}{{\bf table}}
\newtheorem{theorem}{Theorem}[section] 
\newtheorem*{theorem*}{Theorem} 
 
\newtheorem*{observation*}{Observation}
\newtheorem{proposition}{Proposition}[section] 
\newtheorem*{proposition*}{Proposition} 
\newtheorem{corollary}{Corollary}[section] 
\newtheorem*{corollary*}{Corollary}
\newtheorem{conjecture}{Conjecture}[section] 
\newtheorem*{conjecture*}{Conjecture}
\newtheorem{lemma}{Lemma} [section] 
\newtheorem*{lemma*}{Lemma}
 
\newtheorem*{definition*}{Definition} 
 
\newtheorem*{claim*}{Claim}
\usepackage[utf8x]{inputenc}
\usepackage[greek,russian,english]{babel}
\usepackage{amssymb}
\usepackage{graphicx}
\usepackage{hyperref}
\usepackage{color}

\begin{document}

\date{}

\title{The Parameterized Complexity of Graph Cyclability\thanks{The first author was supported by the European Research Council under the European Union's Seventh Framework Programme (FP/2007-2013)/ERC Grant Agreement n. 267959.
The second author was supported by the Foundation for Polish Science (HOMING PLUS/2011-4/8) and 
National Science Center (SONATA 2012/07/D/ST6/02432). The third and the fourth author were co-financed by the E.U. (European Social Fund - ESF) and Greek national funds through the Operational Program ``Education and Lifelong Learning'' of the National Strategic Reference Framework (NSRF) - Research Funding Program: ``Thales. Investing in knowledge society through the European Social Fund''. Emails: \texttt{Petr.Golovach@ii.uib.no}, \texttt{mjk@mimuw.edu.pl}, \texttt{spyridon.maniatis@gmail.com}, \texttt{sedthilk@thilikos.info}}}
\author{Petr A. Golovach\thanks{Department of Informatics, University of Bergen, Bergen, Norway.} \and Marcin Kamiński\thanks{Faculty of Mathematics, Informatics and Mechanics, University of Warsaw, Warsaw, Poland.} \and  Spyridon Maniatis\thanks{Department of Mathematics, National \& Kapodistrian University of Athens, Athens, Greece.}\and  \\ Dimitrios M. Thilikos$^{\S,}$\thanks{ AlGCo project-team, CNRS, LIRMM, Montpellier, France.}}

%

\maketitle

\begin{abstract}
\noindent  The cyclability of a graph is the maximum integer $k$ for which  every $k$  vertices lie on a cycle. The algorithmic version of the problem, given a graph $G$ and a non-negative integer $k,$ decide whether the cyclability of $G$ is at least $k,$ is  {\sf NP}-hard.  We study the parametrized complexity of this problem.
%
We prove that this problem, parameterized by $k,$ is ${\sf co\mbox{-}W[1]}$-hard and that its does not admit a polynomial kernel on planar graphs,  unless ${\sf NP}\subseteq{\sf co}\mbox{-}{\sf NP}/{\sf poly}$.
On the positive side, we give an  {\sf FPT} algorithm for planar graphs that runs in time $2^{2^{O(k^2\log k)}}\cdot n^2$. Our algorithm is based on a series of graph-theoretical results on cyclic linkages in planar graphs.
\end{abstract}


\section{Introduction}

In the opening paragraph of his book \emph{Extremal Graph Theory} Béla Bollobás notes: ``\emph{Perhaps the most basic property a graph may posses is that of being connected. At a more refined level, there are various functions that may be said to measure the connectedness of a connected graph.}'' Indeed, connectivity is one of the fundamental properties considered in graph theory and studying different variants of connectivity provides a better understanding of this property. Many such alternative connectivity measures have been studied in graph theory  but very little is known about their algorithmic properties. The main goal of this paper is to focus on one such parameter -- {\sl cyclability} -- from an algorithmic point of view. Cyclability can be thought of as a quantitative measure of Hamiltonicity, or as a natural ``tuning'' parameter between connectivity and Hamiltonicity. \medskip

\noindent {\it Cyclability}. 
For a positive integer $k$, a graph $G$ is \emph{$k$-cyclable} if every $k$ vertices of $G$ lie on a common cycle; we assume that any graph is $1$-cyclable. 
The \emph{cyclability} of a graph $G$ is the maximum integer $k$ for which $G$ is $k$-cyclable.  
Cyclability is well studied in the graph theory literature. Dirac proved that the cyclability of a $k$-connected graph is at least $k,$ for $k \geq 2$ \cite{Dirac60}. Watkins and Mesner \cite{watkins1967cycles} characterized the extremal graphs for the theorem of Dirac. There is a variant of cyclability restricted only to a set of vertices of a graph. Generalizing the theorem of Dirac, Flandrin et al.  \cite{flandrin2007generalization} proved that if a set of vertices $S$ in a graph $G$ is $k$-connected, then there is a cycle in  $G$ through any $k$ vertices of $S$. (A set of vertices $S$ is $k$-connected in $G$ if a pair of vertices in $S$ cannot be separated by removing at most $k-1$ vertices of $G$.) Another avenue of research is lower-bounds on cyclability of graphs in restricted families. For example, every 3-connected claw-free graph has cyclability at least 6 \cite{plummer2001nine} and every $3$-connected cubic planar graph has cyclability at least 23 \cite{aldred1999cycles}.

Clearly, a graph $G$ is Hamiltonian if and only if its  cyclability equals $|V(G)|$. Therefore, we can think of cyclability as a quantitive measure of Hamiltonicity. A~graph $G$ is \emph{hypohamiltonian} if it is not Hamiltonian but all graphs obtained from $G$ by deleting one vertex are. Clearly, a graph $G$ is hypohamiltonian if and only if its  cyclability equals $|V(G)| -1$. Hypohamiltonian graphs appear in combinatorial optimization and are used to define facets of the traveling salesman polytope \cite{grotschel1977hypohamiltonian}. Curiously, the computational complexity of deciding whether a graph is hypohamiltonian seems to be open. 

To our knowledge no algorithmic study of cyclability has been done so far.
In this paper we initiate this study. For this, we consider the following problem. 

\begin{center}
\fbox{\begin{minipage}{10cm}
\noindent{\sc Cyclability}\\
{\sl Input}:  A graph $G$ and a non-negative integer $k$. \\
{\sl Question}: Is every $k$-vertex set $S$ in $G$ cyclable, i.e., 
is there a cycle  $C$ in $G$ such that $S\subseteq V(C)$?
\end{minipage}}
\end{center}

%
%
%

{\sc Cyclability} with $k=|V(G)|$ is {\sc Hamiltonicity} and {\sc Hamiltonicity}  is  {\sf NP}-complete even for planar cubic graphs~\cite{GareyJT76}. Hence, we have the following. 

\begin{proposition}\label{prop:NP-h}
{\sc Cyclability} is {\sf NP}-hard for cubic planar graphs. 
\end{proposition}
%

\noindent{\it Parameterized complexity}. 
A \emph{parameterized problem} is a language $L\subseteq\Sigma^*\times\mathbb{N}$, where $\Sigma$ is a finite alphabet.
A parameterized problem has as instances pairs $(I,k)$ where $I\subseteq \Sigma^*$ is the main part 
and $k\in\mathbb{N}$ is the parameterized part.   Parameterized Complexity settles the question of whether a parameterized  problem is solvable by  an algorithm (we call it {\sf FPT}-{\em algorithm}) of time complexity $f(k)\cdot |I|^{O(1)}$ where $f(k)$ is a function that does not depend on $n$. If such an algorithm exists, we say that the parameterized problem belongs to the 
class {\sf FPT}. In  a series of fundamental papers (see~\cite{DowneyF95-I,DowneyF95-II,DowneyF93,DowneyF92}), Downey and Fellows defined a series of complexity classes,  such as ${\sf W}[1]\subseteq {\sf W}[2]\subseteq\cdots\subseteq{\sf W}[SAT]\subseteq {\sf W}[P]\subseteq {\sf XP}$
and proposed special types of reductions such that hardness for 
some of the above classes makes it rather impossible that a problem belongs to {\sf FPT}\ (we stress
that ${\sf FPT}\subseteq {\sf W}[1]$).  We mention that {\sf XP} is the class of parameterized problems such that for every $k$ there is an algorithm that solves that problem in time $O(|I|^{f(k)}),$ for some function $f$ (that does not depend on $|I|$).
For more  on parameterized complexity, we refer the reader to \cite{CyganFKLMPPS15} (see also~\cite{DowneyF99}, \cite{FG06}, and \cite{Nie02}). \medskip

\noindent{\it Our results}. In this paper we deal with the parameterized complexity of  {\sc Cyclability} when parameterized by $k$. It is easy to see that {\sc Cyclability} is in {\sf XP}. 
For a graph $G,$ we can check all possible subsets $X$ of $V(G)$ of size $k$. For each subset $X,$ we consider $k!$ orderings of its vertices, and for each sequence of $k$ vertices $x_1,\ldots,x_k$ of $X,$ 
we use the main algorithmic result of Robertson and Seymour in~\cite{RobertsonS95b}, to check whether there are $k$ disjoint paths that join $x_{i-1}$ and $x_i$ for $i\in\{1,\ldots,k\}$ assuming that $x_0=x_k$. We return a \yes-answer if and only if  we can obtain the required disjoint paths for each set~$X$, for some ordering. \medskip

Is it possible that {\sc Cyclability} is  {\sf FPT} when parameterized by $k$? 
Our results are the following:
%

Our first results is that an  ${\sf FPT}$-algorithm for this problem is rather unlikely as it is {\sf co-W[1]}-hard even when restricted to split graphs\footnote{A {\em split graph} is any graph $G$ whose vertex set can be paritioned into  two sets $A$ and $B$ such that $G[A]$ is a complete graph
and $G[B]$ is an edgeless graph.}:

\begin{theorem}
\label{thm:W}
It is {\sf W[1]}-hard to decide for a split graph $G$ and a positive integer $k,$
whether $G$ has $k$ vertices such that there is no cycle in $G$ that contains these $k$ vertices, when the problem is parameterized by $k$.
\end{theorem}

On the positive side we prove that the same parameterized problem admits an ${\sf FPT}$-algorithm when its input is restricted to be a planar graph.

\begin{theorem}\label{mainalgs} 
The {\sc  Cyclability} problem, when parameterized by $k,$ is in {\sf FPT} when  its input graphs  are  restricted to be  planar graphs.
Moreover, the corresponding {\sf  FPT}-algorithm runs in   $2^{2^{O(k^2\log k)}}\cdot n^2$ steps.
\end{theorem}

Actually, our algorithm solves the more general problem where the input comes with a subset $R$ of annotated vertices and the question is whether every $k$-vertex subset of $R$ is cyclable.

Finally, we prove that even for the planar case, the following negative result holds. 

\begin{theorem}\label{thm:no-kernel}
\textsc{Cyclability}, parameterized by $k$, has no polynomial kernel unless ${\sf NP}\subseteq{\sf co}\mbox{-}{\sf NP}/{\sf poly}$ when restricted to cubic planar graphs.
\end{theorem}

The above result indicates that the {\sc  Cyclability} does not follow the kernelization behavior of 
many other problems (see, e.g.,~\cite{BodlaenderFLPST09meta}) for which surface embeddability enables the construction of polynomial kernels. \\

\noindent{\it Our techniques}. Theorem \ref{thm:W} is proved in Section~\ref{hardness} and the proof is a reduction from the standard parameterization of the {\sc Clique} problem. 

The two key ingredients in the proof of Theorem~\ref{mainalgs} are a new, two-step, version of the \emph{irrelevant vertex technique} and a new combinatorial concept of \emph{cyclic linkages} along with a strong notion of {\em vitality} on them (vital linkages played an important role in the Graph Minors series, in~\cite{RobertsonS-XXI} and~\cite{RobertsonS12irre}). The proof of Theorem~\ref{mainalgs} is presented in Section~\ref{profsalg}.
Below, we give a rough sketch of our method. 

We work with a variant of {\sc  Cyclability} in which some vertices (initially all) are colored. We only require that every $k$ colored  vertices lie on a common cycle. 
If the treewidth of the input graph $G$ is ``small'' (bounded by an appropriate  function of $k$), we employ a dynamic programming routine to solve the problem. Otherwise, there exists a cycle in a plane embedding of $G$ such that the graph $H$ in the interior of that cycle is ``bidimensional'' (contains a large subdivided wall) but is still of bounded treewidth. 
This structure permits to distinguish in $H$ a sequence ${\cal C}$ of, sufficiently many, concentric cycles that are all traversed by some, sufficiently many, paths of $H$. Our first aim is to check whether the distribution of the colored vertices in these cycles yields some 
``big uncolored area'' of $H$. In this case we declare  some ``central'' vertex of this area {\em problem-irrelevant} in the sense that 
its removal creates an equivalent instance of the problem. If such an area does not exists, then $R$ is ``uniformly'' distributed 
inside the cycle sequence ${\cal C}$. Our next step is  to set up a sequence of instances of the problem, each corresponding to  the graph ``cropped'' by the interior  of the cycles of ${\cal C}$, where all vertices of a sufficiently big ``annulus'' in it
are now uncolored. As  the graphs of these instances are subgraphs of $H$ and therefore have bounded treewidth, we can get an answer for all of them by performing a sequence of dynamic programming calls (each taking a linear number of steps). At this point, we prove 
that if one of these instances is a \no-instance then initial instance is a \no-instance, so we just report it and stop.
Otherwise, we pick a colored vertex inside the most ``central'' cycle   of  ${\cal C}$ and prove that this 
vertex is {\em color-irrelevant}, i.e., an equivalent instance is created when this vertex is not any more colored.
In any case, the algorithm produces either a solution or some ``simpler'' equivalent instance that either contains a vertex less or a colored vertex less. This permits a linear number of 
recursive calls of the same procedure. To prove that these two last critical steps work as intended, we have to introduce several combinatorial tools. One of them is the notion of strongly vital linkages, a variant of the notion of vital linkages introduced 
in~\cite{RobertsonS-XXI}, which we apply to terminals traversed by cycles instead of terminals linked by paths, as it has been done in~\cite{RobertsonS-XXI}. This notion of vitality permits a significant restriction of the expansion of cycles which certify that sets of $k$ vertices are cyclable and is able to justify both critical steps of our algorithm.  
The proofs of the combinatorial results that support our algorithm are presented in Section~\ref{linkagelems}
and we believe that they have independent combinatorial  importance.

The proof of Theorem~\ref{thm:no-kernel} is given in Section~\ref{hardness} and is based on the cross-composition technique introduced by Bodlaender, Jansen, and Kratsch~in~\cite{BodlaenderJK14}. We show that a variant of the {\sc Hamiltonicity} problem AND-cross-composes to {\sc Cyclabilty}.
\medskip

\noindent{\it Structure of the paper}. The paper is organized as follows. In Section~\ref{defakl} we give a set of definitions that are necessary 
for the presentation of our algorithm. The main steps of the algorithm are presented in Section~\ref{profsalg} and 
the combinatorial results (along with the necessary definitions) are presented in Section~\ref{linkagelems}.
Section \ref{DP}  is devoted to a dynamic programming algorithm for {\sc Cyclability} and Section \ref{hardness} contains the {\sf co-W[1]}-hardness of {\sc Cyclability} for general graphs and the proof of the non-existence of a polynomial kernel on planar graphs.
We conclude with some discussion  and open questions in Section~\ref{discisert}.

%

\section{Definitions and preliminary results}
\label{defakl}
For any graph~$G,$ $V(G)$ (respectively $E(G)$) denotes the \emph{set
of vertices} (respectively \emph{set of edges}) of~$G$.
A graph~$G'$ is a \emph{subgraph} of a graph~$G$ if $V(G')
\subseteq V(G)$ and $E(G') \subseteq E(G)$, and we
denote this by~$G' \subseteq G$.
If $S$ is a set of vertices or a set of  edges of a graph $G,$
graph $G \setminus S$ is the graph obtained from $G$ after
the removal of the elements of $S.$ 
Given a  $S\subseteq V(G)$ we define $G[S]$ as the graph 
obtained from $G$ if we remove from it all vertices not belonging to $S$.
Also, given that $S\subseteq E(G)$, we denote by $G[S]$ the graph whose 
 vertex set is the set of the endpoints of the edges in $S$ and whose 
edge set is $S$.
Given two graphs $G_{1}$ and $G_{2},$ we define  $G_{1}\cap G_{2}=(V(G_1)\cap V(G_{2}),E(G_{1})\cap E(G_{2}))$ and  $G_{1}\cup G_{2}=(V(G_1)\cup V(G_{2}),E(G_{1})\cup E(G_{2}))$. Let ${\cal G}$ be a family of graphs. We denote by $\cupall {\cal G}$  the graph that is the union of all graphs in ${\cal G}$.
For every vertex~$v \in V(G)$, {\em the
 neighborhood} of~$v$ in~$G,$ denoted by~$N_G(v),$ is the subset
 of vertices that are adjacent to~$v,$ and its size is called the
 \emph{degree} of~$v$ in~$G,$ denoted by~${\rm deg}_G(v).$ The \emph{maximum} (respectively \emph{minimum})
 \emph{degree}~${\rm \Delta}(G)$ (respectively  $\delta(G)$) of a graph~$G$ is the maximum (respectively minimum) value taken
 by~${\rm deg}_G(v)$ over~$v\in V(G)$. 
For a set of vertices $U,$ $N_G(U)=\bigcup_{v\in U}N_G(v)\setminus U$. 
 A {\em cycle} of $G$ is a subgraph of $G$ that is connected and all its vertices have degree 2. We call a set of vertices 
 $S\subseteq V(G)$ {\em cyclable} if for some cycle $C$ of $G,$ it holds that $S\subseteq V(C)$.
A cycle $C$ in a graph $G$ is \emph{Hamiltonian} if $V(C)=V(G)$. Respectively, a graph $H$ is \emph{Hamiltonian} if it has a Hamiltonian cycle.  

%
%

\paragraph{Treewidth.}  A {\em tree decomposition} of a graph $G$ is a pair ${\cal D}=({\cal X}, T)$
in which $T$ is a tree and ${\cal X}=\{X_i\mid i\in V(T)\}$ is a family of subsets of $V(G)$  such that:
\begin{itemize}
\item $\bigcup_{i\in V(T)}X_i= V(G)$
 
\item for each edge $e=\{u, v\} \in E(G)$ there  exists an $i\in V(T)$ such that both $u$ and $v$ belong to $X_i$

\item for all $v\in V,$ the set of nodes $\{i\in V(T)\mid v \in X_i\}$ forms a connected subtree of $T$.
\end{itemize}

\noindent The {\em width} of a tree decomposition is $\max\{|X_{i}|\mid i\in V(T)\}-1$. The {\em treewidth} of a graph $G$ (denoted by $\tw(G)$) is the minimum width over all possible tree decompositions of $G$.

\paragraph{Concentric cycles.} 
Let $G$ be a graph embedded in the 
sphere $\Bbb{S}_{0}$ and let ${\cal D}=\{D_1,\ldots,D_r\},$ be a sequence of closed disks in $\Bbb{S}_{0}$. We call ${\cal D}$
{\em concentric} if  $D_{1}\subseteq D_{2}\subseteq \cdots\subseteq D_{r}$ and no point belongs to the boundary of two 
disks in ${\cal D}$. 
We call a sequence ${\cal C}=\{C_{1},\ldots,C_{r}\},$ $r\geq 2,$ of cycles of $G$ {\em concentric}
if there exists a concentric sequence of closed disks  ${\cal D}=\{D_1,\ldots,D_r\},$ such that $C_{i}$ is the boundary of $D_{i},$ $i\in\{1,\ldots,r\}$.  
For $i\in\{1,\ldots, r\},$  we set  $\bar{C}_{i}=D_{i},$ $\mathring{C}_{i}=\bar{C}_{i}\setminus C_{i},$ and  $\hat{C}_{i}=G\cap D_{i}$ (notice that $\bar{C}_{i}$ and $\mathring{C}_{i}$ are sets while $\hat{C}_{i}$ is a subgraph of $G$).  Given $i,j$ with $i\leq j-1,$ we denote by $\hat{A}_{i,j}$ the graph $\hat{C}_{j}\setminus \mathring{C}_{i}$.
Finally, given a $q\geq 1,$ 
we say that a vertex  set $R\subseteq V(G)$ is {\em $q$-dense} in ${\cal C}$ if, for every $i\in \{1,\ldots, r-q+1\},$ $V(\hat{A}_{i,i+q-1})\cap R\neq \emptyset$.

\paragraph{Railed annulus.} 
Let $ r\geq 2$ and $q \geq 1$ be two integers and let $G$ be a graph embedded on the sphere $\Bbb{S}_{0}$. 
A $(r,q)$-{\em railed annulus} in $G$ is a pair $({\cal C},{\cal W})$ such that ${\cal C}=\{C_{1},C_{2},\ldots ,C_{r}\}$ is a sequence of $r$ concentric 
cycles  that are all intersected by a sequence ${\cal W}$ of $q$ paths $W_{1},W_{2},\dots, W_{q}$ (called {\em rails}) 
in such a way that $\cupall{\cal W}\subseteq A_{1,r}$ and the intersection of a cycle and a rail is always connected, that is, it is a (possibly trivial) path (see Figure~\ref{railed_annulus} for an example).
\medskip

\begin{figure}[ht]
    \centering
    \includegraphics[width=0.7\textwidth]{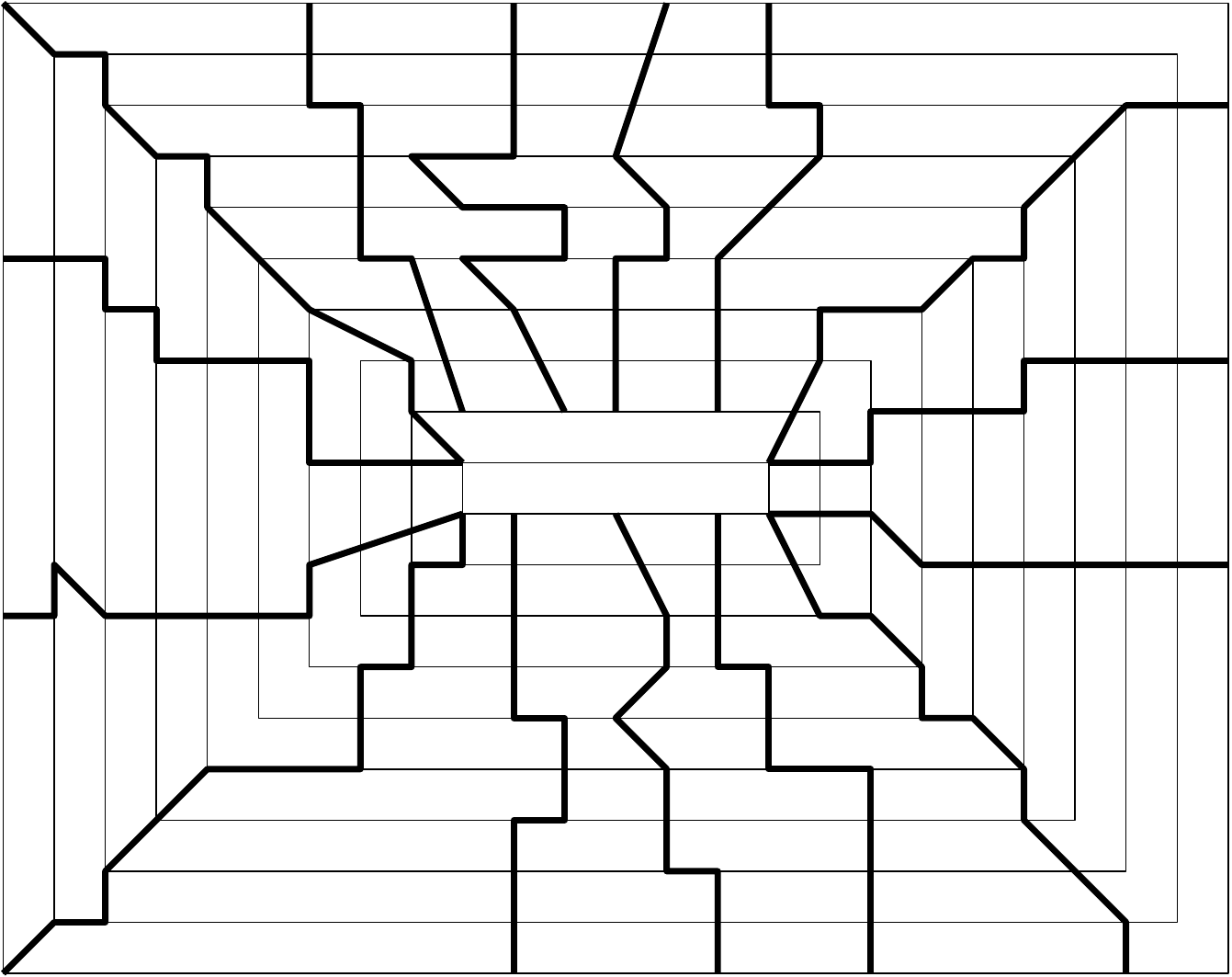}
    \caption{A (10,15)-railed annulus.}
    \label{railed_annulus}
\end{figure}

\medskip

\paragraph{Walls and subdivided walls.} Let $h$ be a integer and $h\geq 1$. A {\em wall of height $h$} is the graph obtained from a $((h+1) \times (2\cdot h+2))$-grid with vertices $(x,y),$ $x\in \{1, \ldots, 2\cdot h+4\},$ $y\in \{1, \ldots, h+1\},$ after the removal of the ``vertical" edges $\{(x,y),(x,y+1)\}$ for odd $x+y,$ and then the removal of all vertices of degree 1. We denote such a wall by $W_h$.
A {\em subdivided wall of height h} is a wall obtained from $W_h$ after replacing some of its edges by paths without common internal vertices 
(see Fig.~\ref{s-wall} for an example). 
The {\em perimeter} $P_W$ of a subdivided wall $W$ is the cycle defined by its boundary. Let $C_{2}=P_{W}$ and let $C_{1}$ be any cycle of $W$ that has no common vertices with $P_{W}$. Notice that ${\cal C}=\{C_{1},C_{2}\}$
is a sequence of concentric cycles in $G$. We define the {\em compass $K_W$ of $W$ in $G$}
as the graph $\hat {C}_{2}$.

\paragraph{Layers of a wall.} Let $W$ be a subdivided wall of height $h\geq 2$. The {\em layers} of $W$ are recursively defined as follows. The first layer, $J_1,$ of $W$ is its perimeter. For $i\in\{2, \ldots, \lfloor \frac{h}{2} \rfloor\},$ the $i$-th layer, $J_i,$ of $W$ is the perimeter of the subwall $W'$ obtained from $W$ by removing its perimeter and  repetitively removing occurring vertices of degree 1. We denote the {\em layer set} of $W$ by ${\cal J}_W=\{J_1, \ldots, J_{\lfloor \frac{h}{2} \rfloor}\}$
\medskip

\begin{figure}[ht]
    \centering
    \includegraphics[width=.92 \textwidth]{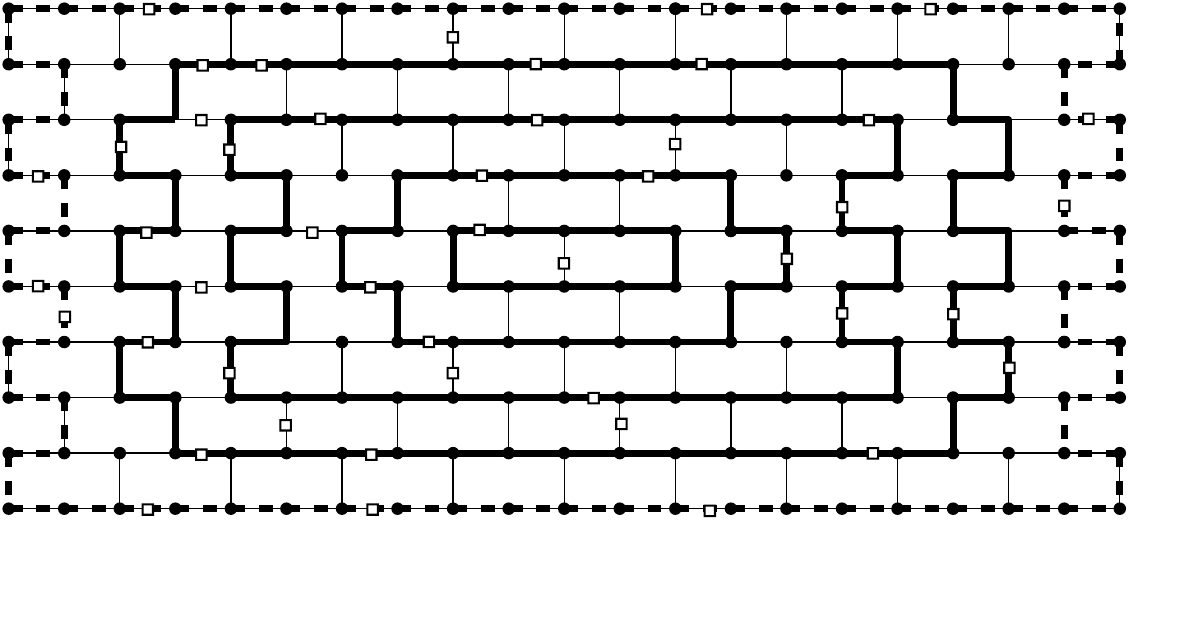}
    \vspace{-0.8cm}
    \caption{A  subdivided wall of height $9$. The white squares represent the subdivision vertices. The bold curves are its layers and the bold-dashed curve is its perimeter.}
    \label{s-wall}
\end{figure}

\medskip

\noindent Given a graph $G$ we denote by $\gw(G)$ the maximum integer $h$ for which $G$ contains a  subdivided wall of height $h$ as a subgraph.
The next lemma follows easily by combining results in~\cite{FominGT11cont}, \cite{GuT10impr}, and \cite{RS91}.
\begin{lemma}
\label{prop:propositionone}
 If $G$ is a planar graph, then $\tw(G)\leq 9\cdot \gw(G)+1$.
\end{lemma}

\section{Vital cyclic linkages}
\label{linkagelems}

\paragraph{Tight concentric cycles.}  Let $G$ be a graph embedded in the 
sphere $\Bbb{S}_{0}$.
A sequence  ${\cal C}=\{C_{1},\ldots,C_{r}\}$ of concentric cycles of $G$  is {\em tight} in $G,$ if
\begin{itemize}
\item  ${C_1}$ is {\em surface minimal}, i.e., there is no closed disk $D$ of $\Bbb{S}_0$ that  is properly contained in $\bar{C_1}$ and whose boundary is a cycle of $G$;
\item for every $i\in\{1,\ldots,r-1\},$ there is no closed disk $D$ such that $\bar{C}_i\subset {D}\subset \bar{C}_{i+1}$
and such that the boundary of $D$ is a cycle of $G$.
\end{itemize}
See Figure~\ref{tight_cycles} for a an example of the tightness definition.

\begin{figure}[ht]
    \centering
    \includegraphics[width=0.7\textwidth]{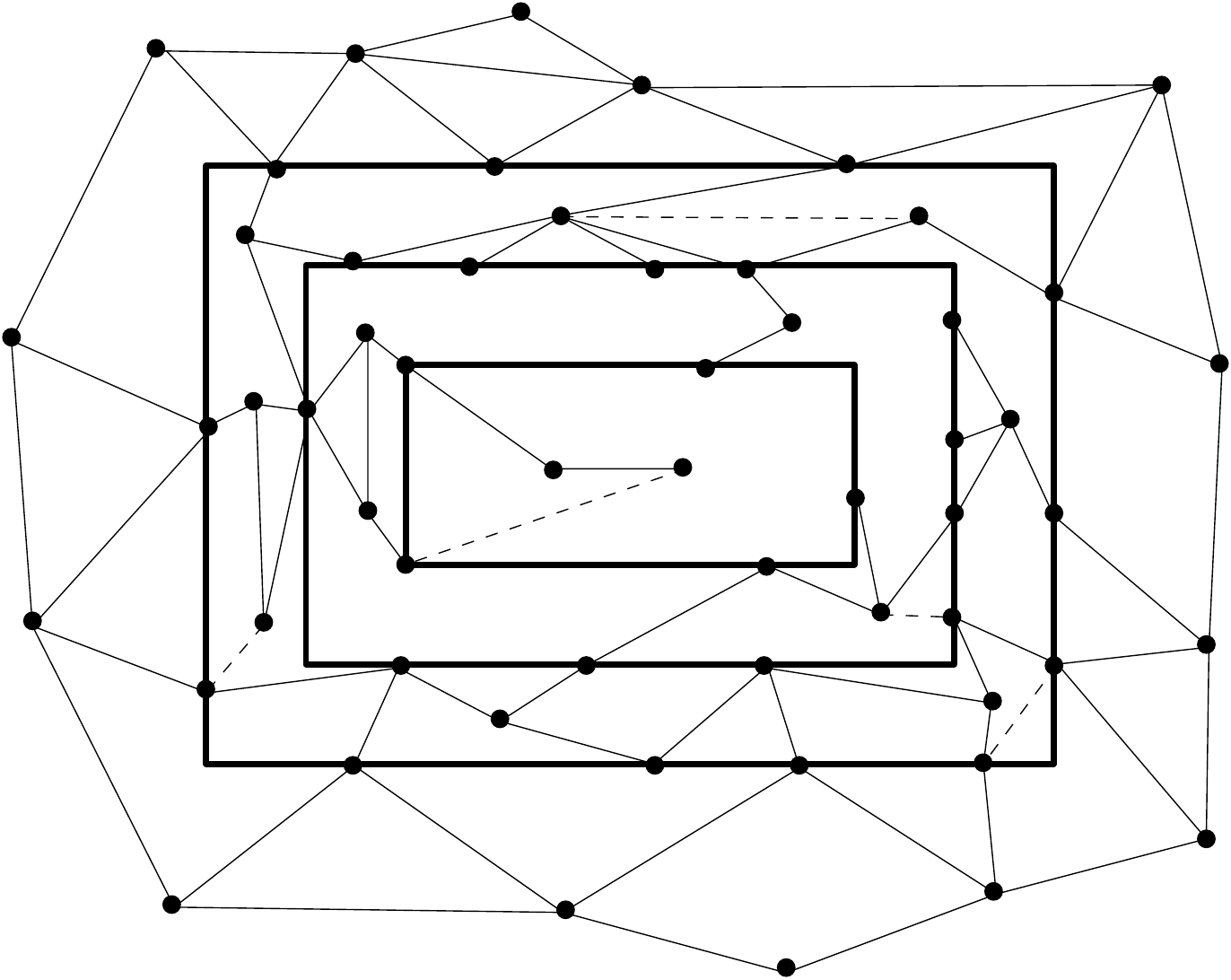}
    \caption{A sequence of three tight concentric cycles. The addition of any of the dashed edges makes the sequence non-tight.}
    \label{tight_cycles}
\end{figure}

\paragraph{Graph Linkages.} 
Let $G$ be a graph.  A  \emph{graph linkage} in $G$ is a pair  ${\cal L}=(H,T)$ 
such that $H$ is a subgraph of $G$ without isolated vertices and $T$ is a subset of the vertices of $H,$ called {\em terminals} of ${\cal L},$
such that every vertex of $H$ with degree different than 2 is contained in $T$.
The set ${\cal P}({\cal L}),$ which we call {\em path set of} the graph linkage ${\cal L},$ contains  all paths of $H$ whose endpoints are in $T$ and do not have any other vertex in $T$. The \emph{pattern} of $L$ is the graph $$(T,\big\{\{s,t\}\mid {\cal P}({\cal L})\text{ contains a path from $s$ to $t$ in $H$}\big\}).$$ Two graph linkages of $G$ are {\em equivalent} if they have the same pattern
and are {\em isomorphic} if their patterns are isomorphic.
A graph linkage ${\cal L}=(H,T)$ is called {\em weakly vital}  (reps.  {\em strongly vital}) in $G$ if $V(H) = V(G)$ and there is no other equivalent (resp. isomorphic) graph linkage that is different from ${\cal L}$.
Clearly, if a graph linkage ${\cal L}$ is strongly vital then it is also weakly vital.
We call a  graph linkage ${\cal L}$ {\em linkage} if its pattern has maximum degree 1 (i.e., it consists of a collection of paths), in which case we omit $H$ and refer to the linkage just by using ${\cal L}$.
We also call a graph linkage ${\cal L}$ {\em cyclic linkage} if its pattern is a cycle. For an example of distinct types of  cyclic linkages, 
see Figure~\ref{fig:Flin}.

Notice that there is a critical difference between equivalence and isomorphism of linkages. 
To see this, suppose that ${\cal L}=(C,T)$ is a cyclic linkage of a graph $G$ and let $A_{G}$ be the set of 
all cyclic linkages that are isomorphic to ${\cal L}$, while $B_{G}$ is the set of 
all cyclic linkages that are equivalent to ${\cal L}$. Notice that the cycles in the cyclic linkages of $A_{G}$
should meet the terminals in the same cyclic order. On the contrary 
the cycles of the cyclic linkages of $B_{G}$ may meet the terminals in any possible cyclic ordering. Consequently $A_{G}\subseteq B_{G}$.
For example, if ${\cal L}=(C,T)$ is the cyclic linkage of the graphs in Figure~\ref{fig:Flin}, then $|A_{G_{1}}|=1$, $|B_{G_{1}}|=1$,
$|A_{G_{2}}|=1$, $|B_{G_{2}}|=12$, $|A_{G_{2}}|=4$, $|B_{G_{3}}|=28$.

\begin{figure}[ht]
\begin{center}
\includegraphics[width=.8\textwidth]{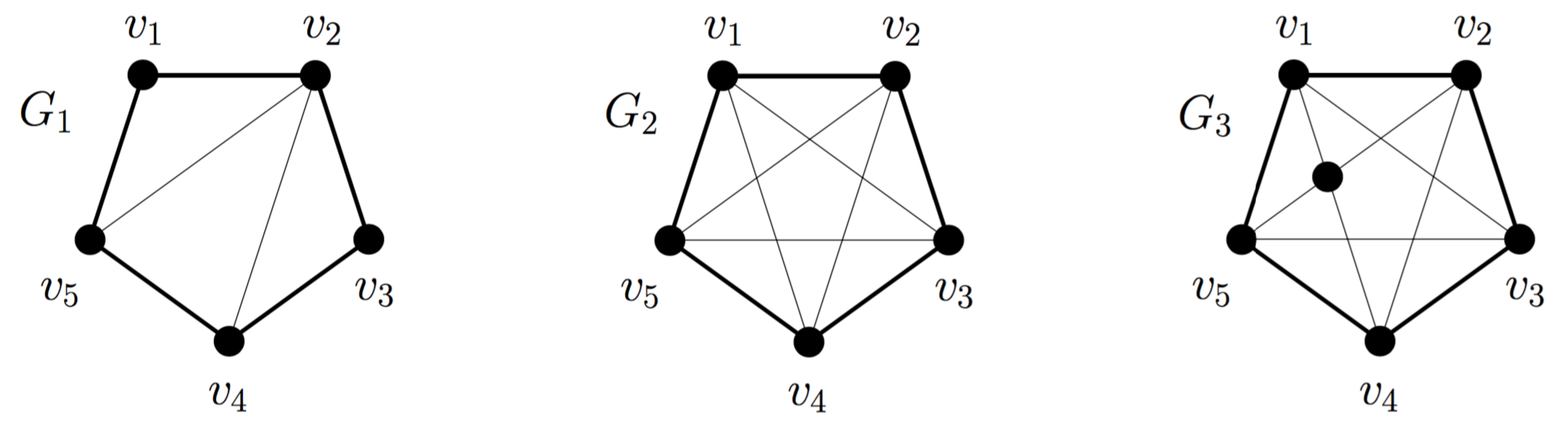}
\end{center}
\caption{Three graphs $G_{1}, G_{2}$, and $G_{3}$. In each graph the bold edges define the 
cycle $C=(\{v_{1},\ldots,v_{5}\},\{\{v_{1},v_{2}\},\ldots,\{v_{4},v_{5}\},\{v_{5},v_{6}\}\})$ where $T=V(C)$.
Consider the cyclic linkage ${\cal L}=\{C,T\}$ where $T=V(C)$.
${\cal L}$ is a weakly vital linkage in $G_{1}$ and $G_{2}$ while it is not a weakly vital linkage  in $G_{3}$.
Moreover, ${\cal L}$ is a strongly vital linkage in $G_{1}$ while it is not a strongly vital linkage  neither in  $G_{2}$  nor in $G_{3}$.}
\label{fig:Flin}
\end{figure}

\paragraph{CGL-configurations.} Let $G$ be a graph embedded on the 
sphere $\Bbb{S}_{0}$. Then we say that a pair ${\cal Q}= ({\cal C}, {\cal L})$ is a {\em CGL-configuration} of {\em depth $r$} if  $\mathcal{C} = \{C_1, \ldots,C_r\}$ is a sequence of concentric cycles in $G,$ ${\cal L}=(H,T)$ is a graph linkage  in $G,$ and $T\cap  V(\hat{C}_r)=\emptyset ,$ i.e., all vertices in the terminals of ${\cal L}$ are outside $\bar{C_r}$. The {\em penetration} of ${\cal L}$ in ${\cal C},$ $p_{\cal C}({\cal L}),$  is the number of cycles of ${\cal C}$ that are intersected by the paths of ${\cal L}$ (when ${\cal L}=(C,S)$ is cyclic we will sometimes refer to the penetration of ${\cal L}$ as the penetration of cycle $C$).
We say that ${\cal Q}$
is {\em touch-free} if for every path $P\in{\cal L},$ the number of  connected components
of $P\cap C_{r}$ is not $1$. See figure~\ref{CLG} for an example of a CGL-configuration.

\begin{figure}[ht]
    \centering
    \includegraphics[width=0.7\textwidth]{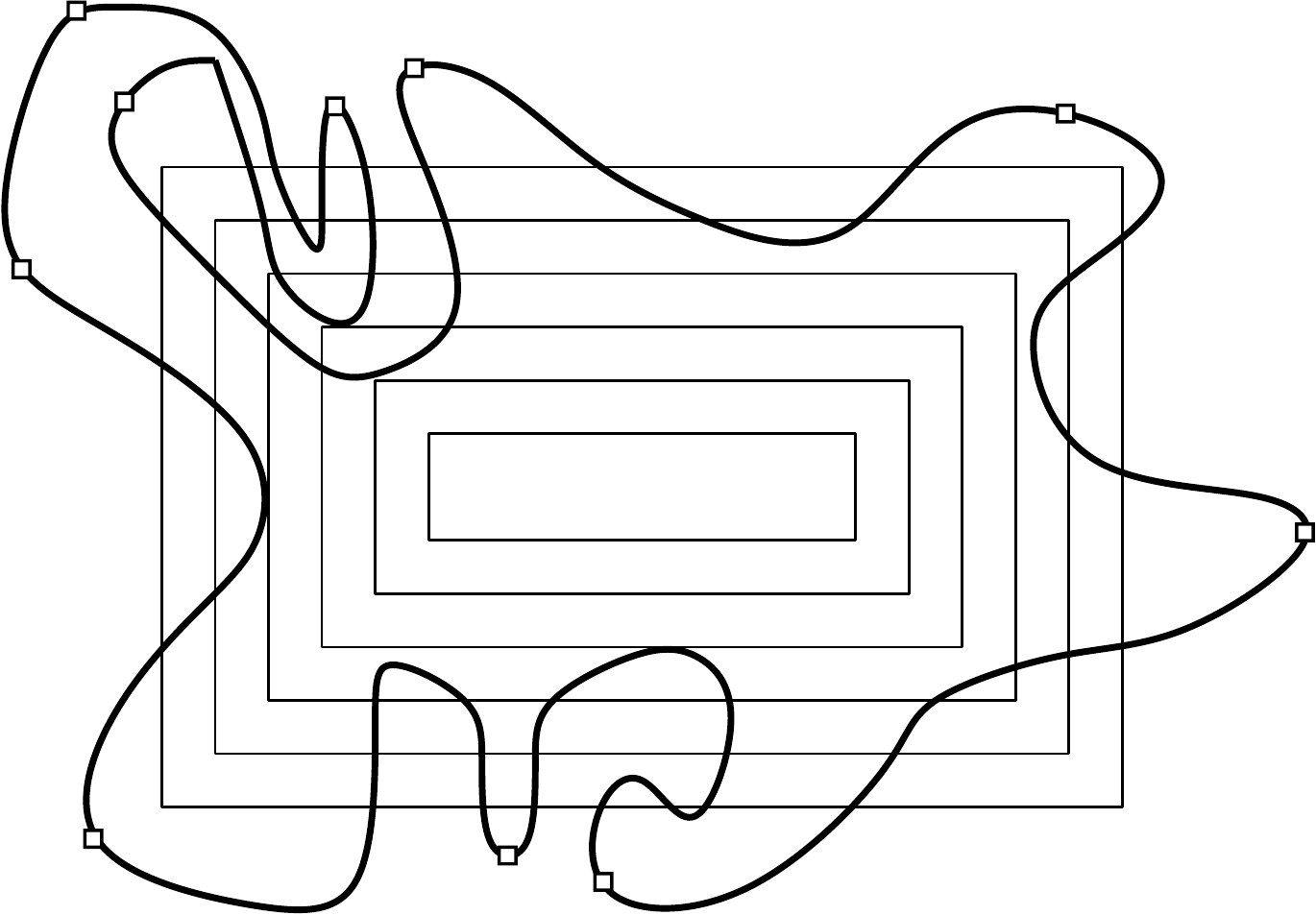}
    \caption{A CLG-configuration ${\cal Q}=({\cal C}, {\cal L})$ with ${\cal L}=(H,T)$. Here, ${\cal C}$ is a sequence of six concentric cycles, $H$ (the {\bf bold} curve) is a cycle (thus ${\cal L}$ is a cyclic linkage) and $T$ is represented by the set of squares. The penetration of ${\cal L}$ in ${\cal C}$ is $4$ and ${\cal Q}$ is touch-free.}
    \label{CLG}
\end{figure}

\paragraph{Cheap graph linkages.} Let $G$ be a graph embedded on the 
sphere $\Bbb{S}_{0},$  let 
${\cal C}=\{C_1,\ldots,C_r\}$ be a sequence of cycles in $G,$
and let  ${\cal L}=(H,T)$ be a graph linkage  
where $T\subseteq V(G \setminus \hat{C}_r)$ (notice that $({\cal C},{\cal L})$ is a  CGL-configuration). 
We define function $c$ which maps
graph  linkages of $G$ to non-negative  integers  such that
 $$c({\cal L}) = |E({\cal L}) \setminus \bigcup_{i \in \{1, \ldots, r\}}{E(C_{i})}|.$$

A graph linkage ${\cal L}$ of $G$ is {\em $\cal{C}$-strongly cheap} (resp. {\em $\cal{C}$-weakly cheap} ), if $T({\cal L}) \cap \hat{C_r} = \emptyset$ and there is no other isomorphic (resp. equivalent)
graph linkage  ${\cal L}'$  such that $c({\cal L}) > c({\cal L}')$. Obviously, if ${\cal L}$ is ${\cal C}$-strongly cheap then it is also ${\cal C}$-weakly cheap.

\paragraph{Tilted grids.}
Let $G$ be a graph. A {\em tilted grid} of $G$ 
is a pair ${\cal U}=({\cal X},{\cal Z})$ where ${\cal X}=\{X_{1},\ldots,X_{r}\}$ 
and ${\cal Z}=\{Z_{1},\ldots,Z_{r}\}$ 
are both sequences of $r\geq 2$ vertex-disjoint paths of $G$ such that 

\begin{itemize}
\item for each $i,j\in\{1,\ldots,r\}$ $I_{i,j}=X_{i}\cap Z_{j}$ is a (possibly 
edgeless) path of 
$G,$
\item for $i\in\{1,\ldots,r\}$ the subpaths $I_{i,1},I_{i,2},\ldots,I_{i,r}$ appear 
in this order in $X_{i}$,
\item for $j\in\{1,\ldots,r\}$ the subpaths $I_{1,j},I_{2,j},\ldots,I_{r,j}$ appear 
in this order in $Z_{j}$,
\item $E(I_{1,1})=E(I_{1,r})=E(I_{r,1})=E(I_{r,r})=\emptyset,$
\item the graph $G_{{\cal U}}^*$ taken from the graph
 $G_{{\cal U}}=(\bigcup_{i\in\{1,\ldots,r\}}X_{i})\cup (\bigcup_{i\in\{1,\ldots,r\}}Z_{i})$ after contracting all edges in $\bigcup_{(i,j)\in\{1,\ldots,r\}^{2}}I_{i,j}$ is isomorphic to the $(r\times r)$-grid.
\end{itemize}
We refer to the cardinality $r$ of ${\cal X}$ (or ${\cal  Z}$) as the {\em capacity} of ${\cal U}$.

\paragraph{Tidy tilted grids.}
Given a plane graph $G$ and a graph linkage ${\cal L}=(H,T)$ of $G$ we say that 
a tilted grid ${\cal U}=({\cal X},{\cal Z})$ of $G$ is an {\em ${\cal L}$-tidy tilted grid} of $G$ if $T\cap D_{\cal U}=\emptyset$ and 
$D_{\cal U}\cap {\cal L}=\cupall {\cal Z}$ where $D_{\cal U}$ is the closed interior of the perimeter of $G_{\cal U}$ (for an example see Figure 6).
%

\paragraph{From graph linkages  to linkages.}
Let $G$ be a graph and let ${\cal L}=(H,T)$ be a graph linkage of  $G$.
We denote by $G_{\cal L}$ the graph obtained by
subdividing all edges of $G$ incident to terminals and then removing the terminals.
Similarly, we define ${\cal L}^*=(H^*,T^*)$ so that 
$H^{*}$ is   the graph obtained by subdividing all edges incident to terminals, removing the terminals, and considering as terminals the subdivision vertices. 
Notice that ${\cal L}^{*}$ is  a linkage of $G_{\cal L}$.
Notice that if ${\cal L}$ is strongly vital then ${\cal L}^{*}$ is not necessarily strongly vital.
However, if ${\cal L}$ is weakly vital, then so is ${\cal L}^{*}$ (see Figure~\ref{linkage} for an example).

\begin{figure}[ht]
    \centering
    \includegraphics[width=0.8\textwidth]{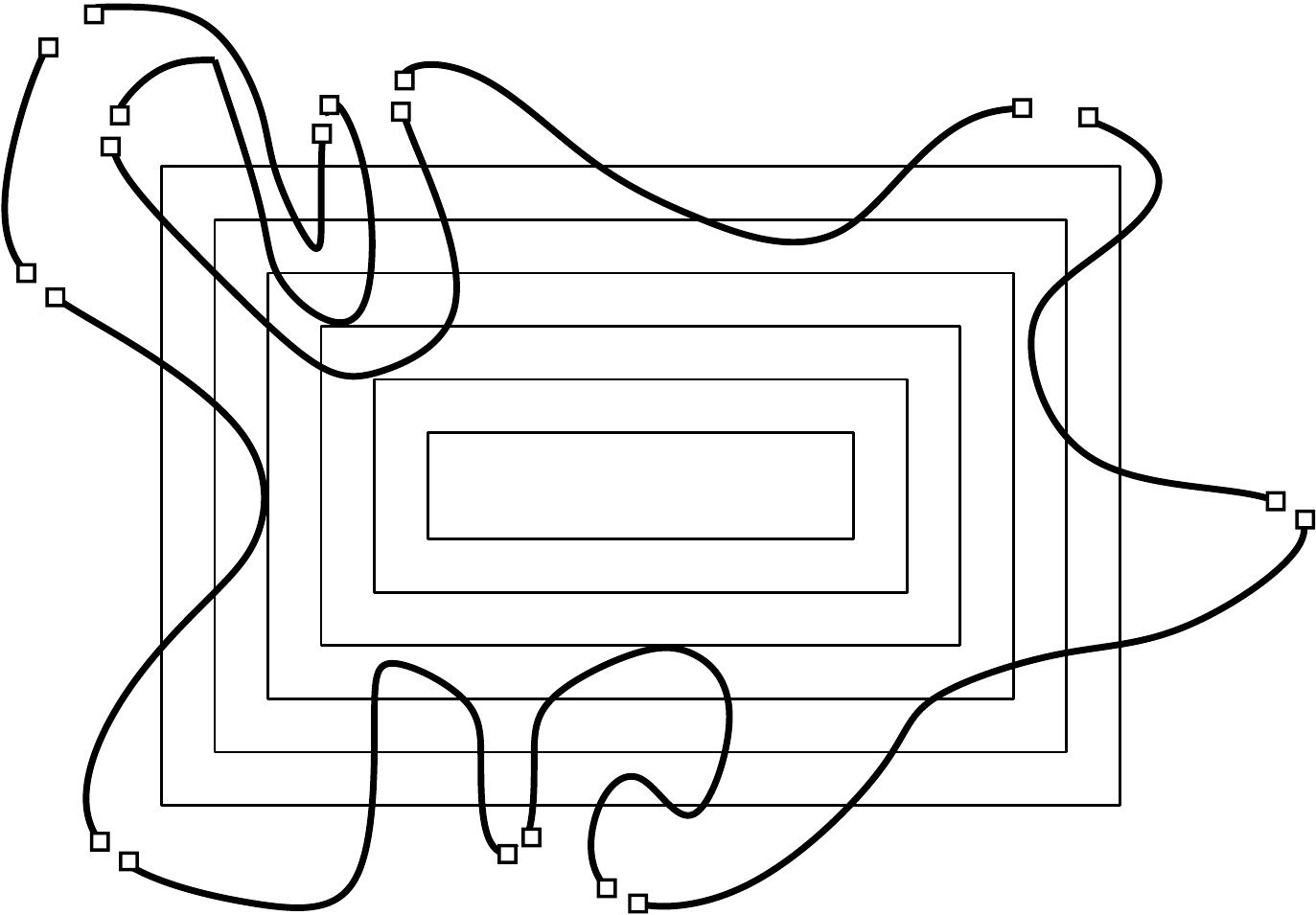}
    \caption{The linkage that corresponds to the cyclic linkage depicted in Figure~\ref{CLG}.}
    \label{linkage}
\end{figure}

%
%

\paragraph{Vertex dissolving.} Let $G$ be a graph and $v \in V(G)$ with $N_G(v)=\{u,w\}$. The operation of \emph{dissolving} $v$ in $G$ is the following: Delete $v$ from $G$ and add edge $\{u,w\}$ to $E(G)$, allowing the existence of multiple edges.
\medskip

The following proposition follows from the combination of Lemma~5, Lemma~6, and Observation~3 of~\cite{AdlerKKLST11tigh}.

\begin{proposition}
\label{prop_cd}
Let $G$ be a graph  embedded on the sphere $\Bbb{S}_0$ and let ${\cal Q}=({\cal C},{\cal L})$ 
 be a  touch-free CGL-configuration of $G,$ where ${\cal C}$ is tight in $G$ and ${\cal L}$ is a  ${\cal C}$-weakly cheap linkage whose penetration in ${\cal C}$ is at least $r$.
Then $G$ contains some ${\cal L}$-tidy tilted grid in $G$ of capacity at least $r/(4\cdot |{\cal P}({\cal L})|)$.
\end{proposition}

\begin{lemma}
\label{lem_4x4grid}
 Let $G$ be a graph embedded on the sphere $\Bbb{S}_0$.
 If $G$ contains a strongly vital cyclic linkage ${\cal L}=(C,T),$ then $G$
 does not contain an ${\cal L}$-tidy tilted grid of capacity $4$.
 \end{lemma}

\begin{proof} Assume that ${\cal L}=(C,T)$ is a strongly vital cyclic linkage in $G$ and that $\Gamma$ is an ${\cal L}$-tidy tilted grid of capacity $4$ in $G$. Let also $\Gamma_4$ be the $(4\times 4)$-grid. Observe that $\Gamma_4$ is the graph that  we get after contracting all edges of $\Gamma$ with at least one endpoint of  degree $2$. We contract $\Gamma$ to $\Gamma_4$ in $G$ and let $G'$ be the resulting graph.\\
Let $V(\Gamma_4)=\{v_{ij} \mid i,j\in\{1, \ldots 4\}\}$ and $E(\Gamma_4)=\{\{v_{ij},v_{i'j'}\} \mid |i-i'|+|j-j'|=1\}$. Observe that $\Gamma_4$ is also an ${\cal L}$-tidy tilted grid of capacity $4$ in $G'$ and that ${\cal L}$ is also strongly vital in $G'$ (if not, then it was not strongly vital in $G$).
Let $H=\Gamma\cup C$ and $H'$ be the contraction of $H$ that we get after contracting all edges of $H$ whose ends have both degree $2$. \\
Let also $H^* = \Gamma_4\cup P_1\cup P_2 \cup P_3\cup P_4$, where for every $i\in\{1,2,3,4\}$, each $P_i$ is a path of length $2$ such that $P_1$ connects $v_{11}$ with $v_{12}$, $P_2$ connects $v_{13}$ with $v_{14}$, $P_3$ connects $v_{41}$ with $v_{44}$ and $P_4$ connects $v_{42}$ with $v_{43}$  (i.e. for every cyclic linkage ${\cal L}=(C,T)$ if we contract all edges of $H=\Gamma\cup C$ whose ends have degree $2$, we get a graph isomorphic to $H^*$ which is a $(4 \times 4)$-grid in addition to some paths that are subgraphs of $C$).\\
It is not hard to confirm that for every possible $H$, its corresponding contraction, $H'$, is isomorphic to $H^*$.
It remains to show that there exists a cyclic linkage ${\cal L'}=(C',T)$ in $G'$, where $C'$ is different from $C$. As $H^*$ is a unique graph (up to isomorphism), a way of rerouting $C$ (in order to obtain a different cyclic linkage) is given in Figure~\ref{4x4}.
\end{proof}

\begin{lemma}
\label{lem:lemma2.3}
Let $G$ be a  graph embedded on the sphere $\Bbb{S}_0$ that is the union of $r\geq 2$ concentric cycles  ${\cal C}=\{C_{1},\ldots,C_{r}\}$ 
and  one more cycle $C$ of $G$. Assume that  ${\cal C}$ is tight in $G,$  $T\cap  V(\hat{C}_r)=\emptyset $, the cyclic linkage ${\cal L}=(C,T)$
is strongly vital in $G$, and its penetration in ${\cal C}$ is $r$. Then $r\leq 16\cdot |T|-1$.
\end{lemma}

\begin{proof}
Let $\sigma: {\cal P}({\cal L})\rightarrow T$ be such
that $\sigma$ is a bijection that maps each path of ${\cal P}({\cal L})$ to one of its endpoints.
For every $i\in\{1,\ldots,r\},$ we 
define ${\cal Q}^{(i)}=({\cal C}^{(i)},{\cal L}^{(i)})$
where ${\cal C}^{(i)}=\{C_{1},\ldots,C_{i}\}$
and ${\cal L}^{(i)}=(C,T^{(i)})$ where \[T^{(i)}=T\setminus \{\sigma(P)\mid P\in P({\cal L}) \mbox{ and } P\cap\hat{C}_{i}=\emptyset\}.\] 

Notice that if some ${\cal Q}^{(i+1)}$ is not touch-free, then $T^{(i)} \leq T^{(i+1)} -1$ (as, by the definition of touch-free configurations, there exists at least one path $P$ in ${\cal P}({\cal L})$ such that $P \cap \hat{C}_{i+1} \neq \emptyset$ but $P \cap \hat{C}_i = \emptyset$).
In the trivial case where every ${\cal Q}^{(i)}$ is not touch-free we derive easily that $r\leq|T|$ and we are done. Otherwise, let ${\cal Q}'=({\cal C}',{\cal L}')$ be the 
touch-free CGL-configuration in $\{{\cal Q}^{(1)},\ldots,{\cal Q}^{(r)}\}$ of the 
highest index, say $i$ (as we excluded the trivial case we have that $i \geq 1$). 
Certainly, ${\cal C}'={\cal C}^{(i)}$ and ${\cal Q}'$ is tight in $G$. Moreover,  ${\cal L}'$ is strongly vital in $G$.
From Lemma~\ref{lem_4x4grid}, $G$ does not contain an ${\cal L}'$-tidy
tilted grid of capacity $4$.  Thus,  $G_{\cal L}$ as well does not  contain an ${\cal L}'^{*}$-tidy (remember how a linkage ${\cal L}^{*}$ is created from a graph linkage ${\cal L}$ after the ``duplication" of the terminals of ${\cal L}$)
tilted grid of capacity $4$.
Recall now that, as ${\cal L}'$ is strongly vital in $G$ it is also weakly vital in $G$ and therefore 
${\cal L}'^{*}$ is weakly vital in ${G}_{{\cal L}'}$. 
Notice  also that ${\cal Q}'^{*}=({\cal C}',{\cal L}'^*)$ is a CGL-configuration of $G_{{\cal L}'}$ where 
${\cal C}'$ is  tight in $G_{{\cal L}'}$. As ${\cal L}'^{*}$ is weakly vital in ${G}_{{\cal L}'},$ then,  by its uniqueness,  ${\cal L}'^{*}$ is ${\cal C}'$-weakly cheap.
Recall that  the penetration of ${\cal L}'$ in ${\cal C}'$ is $r-(i-1)$ and so is the penetration of 
${\cal L}^{\prime *}$ in ${\cal C}'$.
 As ${\cal Q}',$ and therefore ${\cal Q}'^{*}$ as well,  is touch-free we can apply Proposition~\ref{prop_cd} and obtain that  $G_{{\cal L}'}$ contains some ${\cal L}'^{*}$-tidy tilted grid of capacity at least $(r-(i-1))/(4\cdot |{\cal P}({\cal L}'^{*})|) \leq (r-(i-1))/(4 (|{\cal P}({\cal L})|-(i-1))$. 
We derive that $(r-(i-1))/(4 (|{\cal P}({\cal L})|-(i-1))< 4,$ therefore  
$r\leq 16\cdot |{\cal P}({\cal L})|-15(i-1)$
which implies that $r \leq 16\cdot |T| -1$.
\end{proof}
\medskip

\begin{figure}[ht]
    \centering
    \includegraphics[width=0.598\textwidth]{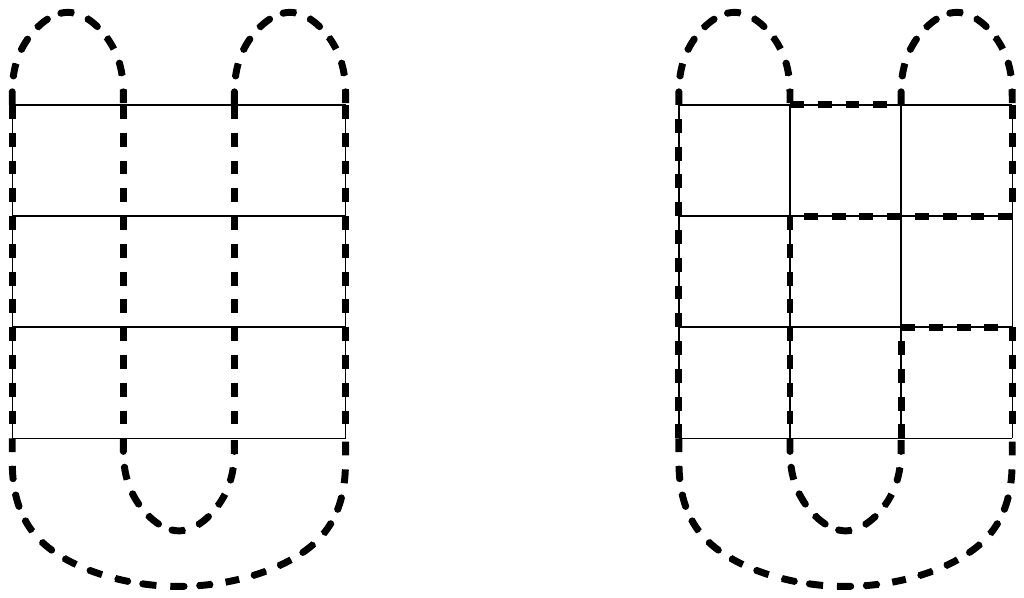}
    \caption{On the left, a simplified ${\cal L}$-tidy $(4\times4)$-grid (corresponding to graph $H^*$) and on the right, a rerouting of the cycle of ${\cal L}$ in the grid.}
    \label{4x4}
\end{figure}

\noindent A corollary of Lemma~\ref{lem:lemma2.3},
with independent combinatorial interest, is the following.

\begin{corollary}
If a plane graph $G$ contains a strongly vital  cyclic linkage  ${\cal L}=(C,T),$
then $\tw(G)=O(|T|^{3/2})$.
\end{corollary}

\noindent Notice that, according to what is claimed in~\cite{AdlerKKLST11tigh}, we cannot restate 
the above corollary for weakly vital linkages, unless we change the bound to be an exponential one.
That way, the fact that treewidth is (unavoidably, due to~\cite{AdlerKKLST11tigh}) exponential to the number of terminals for 
(weakly) vital linkages  is caused by the fact that the ordering of the terminals is predetermined.

\begin{lemma}
\label{lem:cheap2vital}
Let $G$ be a  graph embedded on the sphere $\Bbb{S}_0$ that is the union of $r$ concentric cycles  ${\cal C}=\{C_{1},\ldots,C_{r}\}$ 
and  a Hamiltonian cycle $C$ of $G$. 
Let also   $T\cap  V(\hat{C}_r)=\emptyset $.
If ${\cal L}=(C,T)$ is ${\cal C}$-strongly cheap
then ${\cal L}$ is a  strongly  vital cyclic linkage in $G$.
\end{lemma}

\begin{proof}Assume that ${\cal L}$ is not  strongly  vital in $G,$ i.e., there is a different, isomorphic to ${\cal L}=(C,T),$ cyclic linkage ${{\cal L}'}=(C',T)$ in $G$.
As ${\cal L} \neq {\cal L}'$ we have that $C' \neq C$, therefore there exists an edge $e\in E(C')\setminus E(C)$ (this is because $V(C) = V(C')$ which follows from the strong vitality of ${\cal L}$ in $G$).\\
But, as $E(G)=E(C)\cup \bigcup_{i=1}^{r} E(C_i),$ we derive that $e\in \bigcup_{i=1}^{r} E(C_i)$ (observe that the only way $C'$ can be different from $C$ is by using extra edges from the cycles of ${\cal C}$).\\
 Thus, $|E(C')\cap \bigcup_{i=1}^{r} E(C_i)|> |E(C)\cap \bigcup_{i=1}^{r} E(C_i)|$ and, by the definition of cheap graph linkages,
$c({\cal L})>c({\cal L}'),$ which  contradicts the assumption that ${\cal L}$ is ${\cal C}$-strongly cheap. Therefore, ${\cal L}=(C,T)$ is a strongly vital cyclic linkage in $G,$ as claimed.
\end{proof}

\noindent We are now able to prove the main combinatorial result of this paper.

\begin{lemma}
\label{lem:lemma2.5}
Let $G$ be a plane graph containing some sequence of
concentric cycles $\mathcal{C}=\{C_{1},\ldots,C_{r}\}$.
Let also ${\cal L}=(C,T)$ be a cyclic linkage of $G$ where
$T\cap V(\hat{C}_{r})=\emptyset$. 
If ${\cal L}$ is ${\cal C}$-strongly cheap then the penetration of ${\cal L}$ in ${\cal C}$ is at most  $r\leq 16\cdot |T|-1$.
\end{lemma}

\begin{proof}
Suppose that some path $P\in {\cal P}({\cal L})$ intersects at least
$16\cdot |T|$ cycles of ${\cal C}$. Then, $P$ intersects all cycles in
${\cal C}^*=\{C_{r-16\cdot |T|+1},\ldots,C_{r}\}$.\\
Let $G'$ be the graph obtained by $C\cup \cupall {\cal C}^*$ after dissolving all  vertices of degree $2$ that do not belong to $T$ and let ${\cal L}'=(C',T)$ be the linkage of $G'$ 
obtained from ${\cal L}$ if we dissolve  the same vertices in the paths of ${\cal L}$.  Similarly, by dissolving vertices of degree 2 in  the cycles of ${\cal C}^*$ we obtain a  new sequence of
concentric cycles which, for notational convenience, we denote by $\mathcal{C}'=\{C_{1},\ldots,C_{r'}\},$ where $r'=16\cdot |T|$.\\
The cyclic linkage ${\cal L}'$ is ${\cal C}'$-strongly cheap because 
${\cal L}$ is ${\cal C}$-strongy cheap (it is easy to observe that no edge of $\bigcup_{i=1}^r E(C_r) \setminus E(C)$ belongs to $E(C')$).
Notice that $C'$ is a Hamiltonian cycle of $G'$ and, from Lemma~\ref{lem:cheap2vital}, ${\cal L}'$ is a strongly vital cyclic linkage of $G'$. We also assume that ${\cal C}'$ is tight (otherwise we can replace it by a tight one and observe that, by its uniqueness, ${\cal L}'$ will be cheap to this new one as well). As ${\cal L}'$ is 
${\cal C}'$-strongly cheap and ${\cal C}'$ is tight, from Lemma~\ref{lem:lemma2.3}, $r'\leq 16\cdot |T|-1$; a contradiction.
\end{proof}

\section{The algorithm}
\label{profsalg}

This section is devoted to the proof of Theorem~\ref{mainalgs}.
We consider the following, slightly more general, problem.

\begin{center}
\fbox{\begin{minipage}{11cm}
\noindent{\sc Planar Annotated Cyclability}\\
{\sl Input}:  A plane graph $G,$ a set $R\subseteq V(G),$  and a non-negative integer $k$. \\
{\sl Question}: Does there exist, for every set $S$ of $k$ vertices in $R,$ a cycle $C$ of $G$ such that $S\subseteq V(C)$?
\end{minipage}}
\end{center}

\noindent In this section, 
for simplicity, we refer to {\sc Planar Annotated  Cyclability} as problem {\sc PAC}. 
Theorem~\ref{mainalgs} follows directly from the following lemma.

\begin{lemma}
\label{mainalg}
There is an algorithm that solves  {\sc PAC}   in   $2^{2^{O(k^2\log k)}}\cdot n^2$ steps.
\end{lemma}

The rest of this section is devoted to the proof of Lemma~\ref{mainalg}.

\paragraph{Problem/color-irrelevant vertices.} Let $(G,k,R)$ be an instance of {\sc PAC}. We call a vertex $v \in V(G)\setminus R$ {\em problem-irrelevant}  if $(G,k,R)$ is a \yes-instance if and only if  $(G \setminus v,k,R)$ is a \yes-instance. We call a vertex $v\in R$ {\em color-irrelevant}  when $(G,k,R)$ is a \yes-instance 
 if and only if  $(G,k,R\setminus\{v\})$ is a \yes-instance.

Before we present the algorithm of Lemma~\ref{mainalg}, we need to introduce  three algorithms that are used in it as subroutines.\medskip

\noindent{\bf Algorithm DP}$(G,R,k,q,{\cal D})$\\
\noindent{\sl Input:} A graph $G,$ a vertex set $R\subseteq V(G),$ two non-negative integers $k$ and $q,$ where $k\leq q,$ and a tree decomposition ${\cal D}$  of $G$ of width $q$.\\
\noindent{\sl Output:} An answer whether $(G,R,k)$ is a \yes-instance of  {\sc PAC} or not.\\
\noindent{\sl Running time:} $2^{2^{O(q\cdot \log q)}}\cdot n$.
\medskip

Algorithm {\bf DP}  is based on dynamic programming on tree decompositions of  graphs. The  technical details are presented in Section~\ref{DP}.
\medskip

\noindent{\bf Algorithm Compass}$(G,q)$\\
\noindent{\sl Input:} A planar graph $G$ and a non-negative integer $q.$\\
\noindent{\sl Output:} Either a tree decomposition of $G$ of width at most $18q$
or  a subdivided wall $W$ of $G$ of height  $q$  and a tree decomposition ${\cal D}$  of  the compass $K_{W}$ of $W$ 
of width at most $18q$. \\
\noindent{\sl Running time:} $2^{q^{O(1)}}\cdot n$.
\medskip

We describe algorithm  {\bf Compass}   in Subsection~\ref{compassproof}.\medskip

\noindent{\bf Algorithm concentric\_cycles}$(G,R,k,q,W)$\\
\noindent{\sl Input:} A planar graph $G,$ a set $R\subseteq V(G),$ a non-negative integer $k,$ and a subdivided wall $W$ of $G$ of height at least $392k^2+40k$.\\
\noindent{\sl Output:} Either  a problem-irrelevant vertex $v$ 
or a sequence  ${\cal C}=\{C_{1},C_{2},\ldots ,C_{98k+2}\}$ of concentric cycles of $G,$  with the following properties:

\begin{enumerate}
\item[(1)]  $\bar{C}_1\cap R\neq \emptyset$.
\item[(2)] The set $R$ is $32k$-dense in ${\cal C}$.
\item[(3)] There exists a sequence  ${\cal W}$ of $2k+1$ paths in  $K_W$ such that $({\cal C},{\cal W})$ is a $(98k+2, 2k+1)$-railed annulus.
\end{enumerate}
\noindent{\sl Running time:} $O(n)$.\medskip

We describe Algorithm  {\bf  concentric\_cycles}  in Subsection~\ref{accproof}.
We now use the above three algorithms to describe the main algorithm of this paper which is the following.\medskip

\noindent{\bf Algorithm Planar\_Annotated\_Cyclability}$(G,R,k)$\\
\noindent{\sl Input:} A planar graph $G,$ a set $R\subseteq V(G),$ and  a non-negative integer $k$.\\
\noindent{\sl Output:} 
An answer whether $(G,R,k)$ is a \yes-instance of {\sc PAC} or not. \\
\noindent{\sl Running time:} $2^{2^{O(k^2\log k)}}\cdot n^2$.

\begin{itemize}
\item[]\hspace{-6mm}[{\sl Step {\it 1}.}]  Let  $r=98k^2+2k,$ $y=16k,$  and $q=2y+4r$.
If {{\bf Compass}($G,q)$} returns a tree decomposition of $G$ 
of width $w=18q,$  then  return {\bf DP}$(G,R,k,w)$ and stop. Otherwise, the algorithm {{\bf Compass}($G,q)$} returns a  subdivided wall $W$
of $G$  of height $q$ and a tree decomposition  ${\cal D}$  of  the compass $K_{W}$ of $W$ of width at most $w$.
\medskip

\item[]\hspace{-6mm}[{\sl Step {\it 2}.}]  If the algorithm {\bf concentric\_cycles}$(G,R,k,q,W)$ returns a problem-irrelevant vertex $v,$ then return {\bf Planar\_Annotated\_Cyclability}$(G\setminus v,R\setminus v,k)$ and stop.  Otherwise, it returns a  
sequence  ${\cal C}=\{C_{1},C_{2},\ldots ,C_r\}$ of concentric cycles of $G$ with the properties (1)--(3).\\

\item[]\hspace{-6mm}[{\sl Step {\it 3}.}]  For every $i\in\{1,\ldots, r-98k-2\}$ let $w_{i}$ be a vertex in $\hat{A}_{i+k,i+33\cdot k}\cap R$ (this vertex exists as, from property (2),  $R$ is $32 k$-dense in ${\cal C}$), let $R_{i}=(R\cap V(\hat{C}_{i}))\cup\{w_i\},$ and 
let ${\cal D}_{i}$ be a tree decomposition of $\hat{C}_{i}$ of width at most $w$ -- this tree decomposition 
can be constructed in linear time from ${\cal D}$  as each $\hat{C}_{i}$ is a subgraph of $K_{W}$.
\medskip

\item[]\hspace{-6mm}[{\sl Step {\it 4}.}] 
If, for some $i\in\{1,\ldots, r-98k-2\},$ the algorithm ${\bf DP}(\hat{C}_{i},R_{i},k,q,{\cal D}_{i})$ returns a negative answer, then return a negative answer
and stop. Otherwise  return {\bf Planar\_Annotated\_Cyclability}$(G,R\setminus v,k)$  where $v\in V(\hat{C_1})\cap R$ (the choice of $v$ is possible due to property (1)).
\end{itemize}

\noindent {\it Proof of Lemma~\ref{mainalg}.}
The only non-trivial step in the above algorithm is Step 4. Its correctness
follows  from Lemma~\ref{lem:color_irrelevant}, presented in Subsection~\ref{lemlem}.

We now proceed to the analysis of the running time of the algorithm.
Observe first that the call of {\bf Compass}($G,q)$ in Step 1 takes $2^{k^{O(1)}}\cdot n$
steps and, in the case that a tree decomposition is returned, the ${\bf DP}$ requires 
$2^{2^{O(k^2\log k)}}\cdot n$ steps. For Step 2, the algorithm {\bf concentric\_cycles} takes 
$O(n)$ steps and if it returns a problem-irrelevant vertex, then the whole algorithm is applied again for 
a graph with one  vertex less. Suppose now that 
Step 2 returns a  sequence  ${\cal C}$ of concentric cycles of $G$ with the properties (1)--(3).
Then the algorithm ${\bf DP}$ is called $O(k^2)$ times and this takes in total $2^{2^{O(k^2\log k)}}\cdot n$ steps.
After that, the algorithm either concludes to a negative answer or is called again with one vertex less in the set $R$.
In both cases where the algorithm is called again we have that the quantity $|V(G)|+|R|$ is becoming smaller. 
This means that the recursive calls of the algorithm cannot be more than $2n$. 
Therefore the total running time is bounded by $2^{2^{O(k^2\log k)}}\cdot n^2$  as required.

\subsection{The algorithm {\bf Compass}}
\label{compassproof}

Before we start the description of algorithm \noindent{\bf  Compass}
we present a result that  follows from Proposition~\ref{prop:propositionone}, the algorithms in~\cite{PerkovicR00anim} and~\cite{Bodlaender96alin}, 
and the fact that finding a subdivision of a planar $k$-vertex graph $H$ that has maximum degree 3 in a graph $G$  can be done, using dynamic programming,
in $2^{O(k\cdot \log k)}\cdot n$ steps (see also~\cite{AdlerDFST10fast}).

%

\begin{lemma}
\label{big_grid_small_width}
There exists an algorithm $A_{1}$ that, given a graph $G$ and an integer $h,$ 
outputs either a tree decomposition of $G$ of width at most $9h$
or a subdivided wall of $G$ of height $h$. This algorithm runs in $2^{h^{O(1)}}\cdot n$ steps.
\end{lemma}

\paragraph{Description of algorithm \noindent{\bf  Compass}}
%
%
%
%
We use a routine, call it $A_{2}$, that receives as input a subdivided wall $W$ of $G$ with height equal to some even number $h$ and outputs a subdivided wall $W'$ of $G$  such that $W'$ has height $h/2$ and $|V(K_{W'})|\leq |V(G)|/4$. 
$A_{2}$ uses the fact that, in $W,$ there are  4 vertex-disjoint 
subdivided subwalls of $W$ of height $h/2$. Among them, $A_{2}$ outputs the  one with the minimum number of vertices and this can be done  in $O(n)$ steps. 
The algorithm  
{\bf  Compass} uses as subroutines the routine $A_{2}$ and the algorithm $A_{1}$ of Lemma~\ref{compassproof}.

\begin{tabbing}
\noindent{\bf Algorithm Compass}$(G,q)$\\
\mbox{[}{\sl Step {\it 1}.}\mbox{]}  \=  if \= $A_{1}(G,2q)$ outputs a tree decomposition ${\cal D}$ of $G$ with\\
\> \> width at most $18q$  then return ${\cal D},$\\
\> \> otherwise it outputs   a subdivided wall $W$ of $G$ of height $2q$\\
\mbox{[}{\sl Step {\it 2}.}\mbox{]}   \=  Let $W' = A_{2}(W)$\\
\>  \> if \=$A_{1}(K_{W'},2q)$ outputs  a tree decomposition ${\cal D}$ of \\
\> \> \> $K_{W'}$ of width at most $18q$  then return $W'$ and  ${\cal D},$\\
\>  \> \> otherwise $W\leftarrow {W'}$ and  go to Step 2.
\end{tabbing}


{\noindent}Notice that, if $A$ terminates after the first execution of Step 1, then it outputs a tree decomposition of $G$ of width at most $18q$. 
Otherwise, the output is a subdivided wall $W'$ of height $q$ in $G$ and a tree decomposition of $K_{W'}$ of width at most 
$18q$ (notice that as long as this is not the case, the algorithm keeps returning to step 2).
 The application of routine $A_2$ ensures that the number  of vertices of every new $K_{W}$ is at least four times smaller than the one of the previous one.
Therefore,  the $i$-th call of the algorithm $A_{1}$ 
requires $O(2^{h^{O(1)}}\cdot \frac{n}{2^{2(i-1)}})$ steps. As $\sum_{i=0}^{\infty} \frac{1}{2^{2i}}=O(1),$
algorithm {\bf Compass} has the same running time as algorithm $A_1$.

\subsection{The Algorithm {\bf  concentric\_cycles}}
\label{accproof}

We need to introduce two lemmata.
The first one is strongly based on  the combinatorial Lemma~\ref{lem:lemma2.5} that is the main result of Section~\ref{linkagelems}.

\begin{lemma}
\label{lem:problem_irrelevant}
Let $(G,R,k)$ be an instance of {\sc PAC} and let ${\cal C}=\{C_1, \ldots, C_r\}$ be a sequence of concentric cycles in $G$ such that  $V(\hat{C}_r) \cap R=\emptyset$. If $r\geq 16\cdot k,$ then all vertices in $V(\hat{C}_{1})$  are problem-irrelevant.
\end{lemma}

\begin{proof} We observe that for every vertex $v\in V(G),$ if $(G\setminus v,R,k)\in \Pi$ then $(G,R,k)\in \Pi$ because $G\setminus v$ is a subgraph of $G$ and thus every cycle that exists in $G\setminus v$ also exists in $G$.\\
Assume now that $(G,R,k)\in \Pi,$ let $v\in V(\hat{C}_1),$ and let $S\subseteq R,$ $|S|\leq k$. 
We will prove that there exists a cycle in $G\setminus v$ containing all vertices of $S$.\\
As $(G,R,k)\in \Pi,$ there is a cyclic linkage ${\cal L}=(C,S)$ in $G$. If $v\notin V(C),$ then $C$ is a subgraph of $G\setminus v$  and we are done. Else, if $v\in V(C),$ let ${\cal L}'=(C',S)$ be a ${\cal C}$-weakly cheap cyclic linkage in the graph $H=G[V(C)\cup \lp\bigcup_{i=1}^{r} V(C_i)\rp],$ and assume that $v\in V(C')$ too. Then $C'$ meets all cycles of ${\cal C}$ and its penetration in ${\cal C}$ is more than $16\cdot|S|,$ which contradicts Lemma \ref{lem:lemma2.5}.\\
 Thus, $v\notin V(C')$ implying that there exists a cyclic linkage with $S$ as its set of terminals that does not contain $v$. As $S$ was arbitrarily chosen, vertex $v$ is problem-irrelevant.
\end{proof}

\begin{lemma}
\label{lem:railed_annulus}
Let $y,r,q,z$ be positive integers such that $y+1\leq z\leq r,$ $G$ be a graph embedded on $\Bbb{S}_{0}$ and let $R\subseteq V(G)$ be the set of annotated vertices of $G$.
Given a subdivided wall $W$ of height $h=2\cdot \max\{y,\lceil \frac{q}{8} \rceil\}+4r$ in $G$
then either $G$ contains a sequence  ${\cal C}'=\{C'_{1},C'_{2},\ldots ,C'_{y}\}$ of concentric cycles
such that $V(\hat{C}'_{y})\cap R=\emptyset$ or a sequence ${\cal C}=\{C_{1},C_{2},\ldots ,C_{r}\}$ of concentric cycles such that: 
\begin{enumerate}
\item  $\bar{C}_1\cap R\neq \emptyset$. 

\item $R$ is $z$-dense in ${\cal C}$.

\item  There exists a collection ${\cal W}$ of $q$ paths in $K_W,$ such that $({\cal C}, {\cal W})$ is a $(r,q)$-railed annulus in $G$.
\end{enumerate}  

Moreover, a sequence ${\cal C}'$ or ${\cal C}$ of concentric cycles as above can be constructed in $O(n)$ steps.\smallskip
\end{lemma} 

\begin{proof} 
Let $p=\max\{y,\lceil\frac{q}{8}\rceil\}$. We are given a 
subdivided wall $W$ of height $h=2p+4r$ and we define ${\cal C}=\{C_{1},\ldots,C_{r}\}$
such that $C_{i}=J_{\frac{h}{2}-p-2i+2}, i\in\{1,\ldots,r\}$.
Notice that there is a collection ${\cal W}$ of $8p$ vertex disjoint paths in $W$ such that 
$({\cal C},{\cal W})$ is a $(r,q)$-railed annulus. 
If $\bar{C}_1\cap R= \emptyset,$ then ${\cal C}'=\{J_{\frac{h}{2}},\ldots,J_{\frac{h}{2}+y-1}\}$
is a sequence of concentric cycles where $\bar{J}_{\frac{h}{2}+y-1}\subseteq \mathring{C_1}$ and we are done. Otherwise, we have that ${\cal C}$ satisfies property {\it 1}.\\
Suppose now that Property {\it 2} does not hold for ${\cal C}$.
Then, there exists some $i\in\{1,\ldots,r\}$ such that $A_{i,i+z-1}\cap R=\emptyset$.
Notice that $A_{i,i+z-1}$ contains $2z-1> 2y$ layers of $W$ which are crossed by at least $2y$
of the paths in ${\cal W}$  (these paths certainly exist as $2y<8p$). This implies the existence of a wall of height $2y$ in  $A_{i,i+z-1}$
which, in turn contains a sequence ${\cal C}'=\{C_{1}',\ldots,C_{y}'\}$ of concentric cycles.
As  $\bar{C}_{y}'\subseteq A_{i,i+z-1}$ we have that $V(\hat{C}'_{y})\cap R=\emptyset$ and we are done.
It remains to verify property {\it 3} for ${\cal C}$. This follows directly by including in ${\cal W}'$
any $q\leq 8p$ of the disjoint paths of ${\cal W}$. Then $({\cal C},{\cal W}')$
is the required  $(r,q)$-railed annulus. It is easy to verify that all steps of this proof can be 
turned to an algorithm that runs in linear, on $n,$ number of steps.
\end{proof}

\paragraph{Description of algorithm
 {\bf  concentric\_cycles} }
 This algorithm first applies the algorithm 
 of Lemma~\ref{lem:railed_annulus} for $y=16k,$ $r=98k^2+2k,$ $q=2k+1,$ and $z=32k$.
 If the output is a sequence  ${\cal C}'=\{C'_{1},C'_{2},\ldots ,C'_{y}\}$ of concentric cycles
such that $V(\hat{C}'_{y})\cap R=\emptyset$, then it returns 
a vertex $w$ of $\hat{C}_{1}'$. As $V(\hat{C}_r) \cap R=\emptyset,$ 
Lemma~\ref{lem:problem_irrelevant} implies that $w$ is problem-irrelevant. 
If the  output is a sequence ${\cal C}$ the it remains to observe that 
conditions {\em 1}--{\em 3} match the specifications of  algorithm {\bf  concentric\_cycles}.

\subsection{Correctness of algorithm {\bf Planar\_Annotated\_Cyclability}}
 \label{lemlem}
%
%

As mentioned in the proof of Lemma~\ref{mainalg}, the main step -- [{\it step 4}] -- of 
algorithm {\bf Planar\_Annotated\_Cyclability} is based on Lemma~\ref{lem:color_irrelevant} below.

\begin{lemma}
\label{lem:color_irrelevant}
Let $(G,R,k)$ be an instance of problem {\sc PAC} and let  $b=98k+2$ and $r=98k^2+2k$. Let also $({\cal C},{\cal W})$ be a $(r,2 k+1)$-railed annulus in $G,$ where ${\cal C}=\{C_1, \ldots C_r\}$ is a
sequence of concentric cycles such that 
$\hat{C}_1$ contains some vertex $v\in R$ and that $R$ is $32k$-dense in ${\cal C}$. For every $ i\in\{1, \ldots, r-b\}$ let $R_{i}=(R\cap V(\hat{C}_{i}))\cup\{w_i\},$ where $w_i\in V(\hat{A}_{i+k+1,33k+i+1})\cap R$.
If $(\hat{C}_{i+b},R_i,k)$ is a \no-instance of ${\rm \Pi},$ for some  $i\in\{1, \ldots, r-b\},$ then $(G,R,k)$ is a \no-instance of {\sc PAC}. Otherwise  vertex $v$ is {\em color-irrelevant}.
\end{lemma}

We first prove the following lemma, which reflects the use of the rails of a railed annulus and is crucial for the proof of Lemma \ref{lem:color_irrelevant}.
\medskip

\begin{figure}[ht]
    \centering
    \includegraphics[width=.7\textwidth]{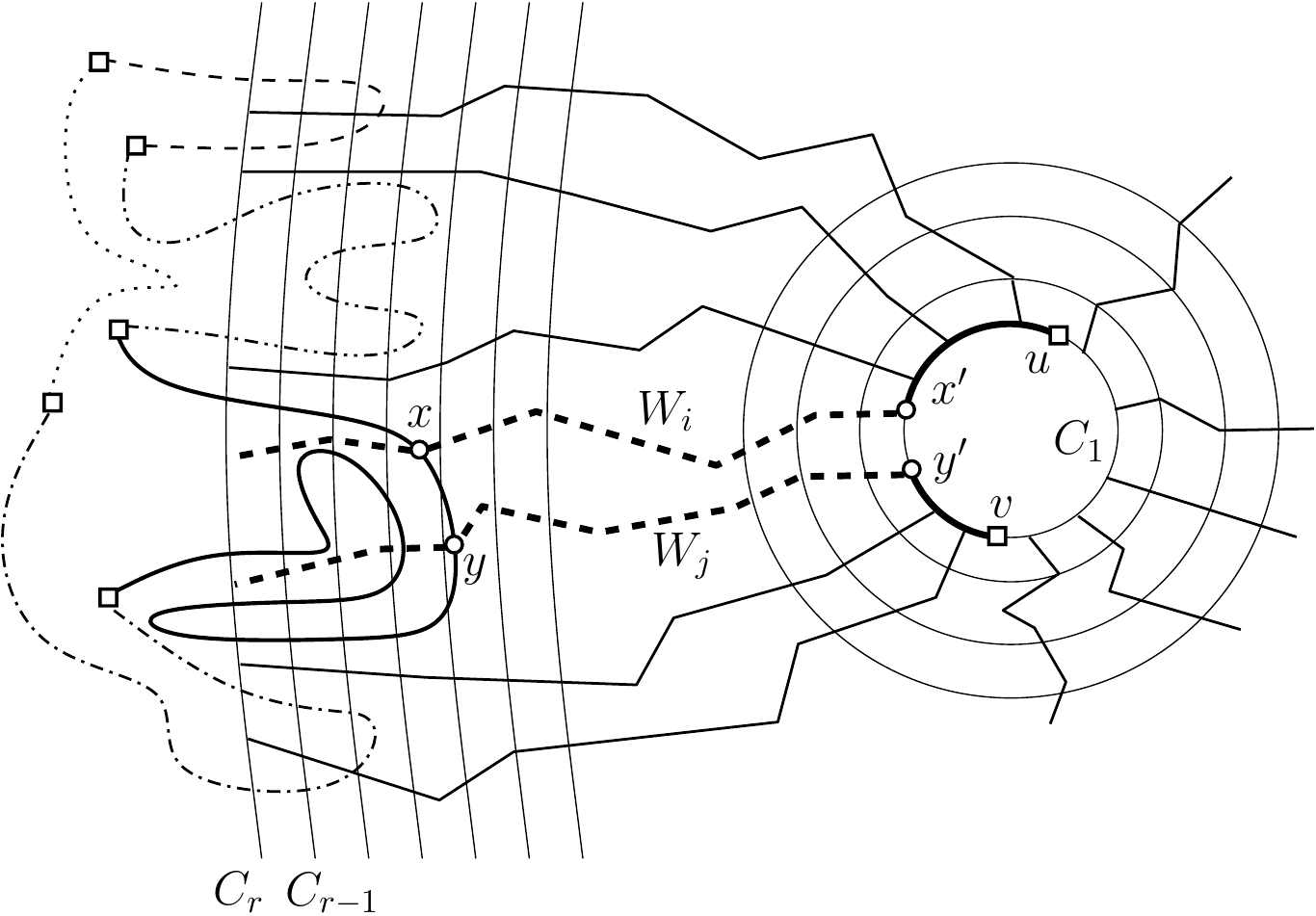}
    \caption{Visualization of proof of Lemma~\ref{lem:rails}, case 1. The different lining on the parts of the cycle at the left indicates the different colors of these paths.  
    }
    \label{lemma11_1}
\end{figure}

\medskip

\begin{lemma}
\label{lem:rails}
Let $G$ be a graph embedded on the sphere $\Bbb{S}_0,$ $r,k$ be two positive integers such that $r\geq16 k$,  and $({\cal C},{\cal W})$ be an $(r,2 k+1)$-railed annulus of $G$ with ${\cal C}=\{C_1, \ldots, C_r\}$ being its sequence of concentric cycles, ${\cal W}=\{W_1, \ldots, W_{2k+1}\}$ its rails. Let also $S\subseteq V(G)$ such that $S\cap \hat{C}_r=\emptyset$ and $|S|=k$. Then for every two vertices $u,v\in V(C_1),$ if there exists a cyclic linkage ${\cal L}=(C,S),$ with penetration $k+1\leq p_{\cal C}({\cal L})\leq r-1,$ in $G,$ then there exists a path $P_{u,v}$ with ends $u$ and $v$ that meets all vertices of $S$.
\end{lemma}

\begin{proof} Let $\{s_1, \ldots, s_k\}$ be an ordering of the set $S$ and let $f_{\cal L}:{\cal L} ({\cal P}) \rightarrow \{1, \ldots, k\}$ be a function such that  for every $i\in\{1, \ldots k-1\}$, $f_{\cal L}(P)= i$ if the endpoints of $P$ are $s_{i}$ and $s_{i+1}$ and $f_{\cal L}(P^*)=k$ for the unique path $P^* \in {\cal P}({\cal L})$ whose endpoints are $s_k$ and $s_1$. \\
 Moreover, as $W_i$ is a path with endpoints $w'_i\in V(C_1)$ and $w''_i \in V(C_r),$ we define the ordering $\{w'_i, \ldots, w''_i\}$ of $V(W_i)$ and call it the natural ordering of $W_i$. Furthermore, for every $W_i\in {\cal W},$ let $m_{\cal L}(W_i)=f_{\cal L}(P)$ if $P$ is the first path (with respect to the natural ordering of $W_i$) of ${\cal P}({\cal L})$ that $W_i$ meets and $m_{\cal L}(W_i)=0$ if $W_i$ does not meet $C$.  \\
Let $C_j\in {\cal C}$. We pick an arbitrary vertex $v_{1}^{j}\in V(C_j)$ and order $V(C_j)$ starting from $v_{1}^{j}$ and continuing in clockwise order. Let $\{v_1^j, \ldots, v_{|V(C_j)|}^j\}$ be such an ordering of the vertices of $C_j$. We assign to each vertex of $v_i^j\in C_j$ a ``color" from the set $\{0, \ldots k\}$ as follows: $c_{\cal L}(v_i^j)=0$ if $v_i^j\notin V(C_j)\cap V(C)$ and $c_{\cal L}(v_i^j)=f_{\cal L}(P)$ if $v_i^j\in V(C_j)\cap V(P),$ where $P\in {\cal P}({\cal L})$.\\
For the rest of the proof, if $P_0$ is a path, $P_0(v,w)$ is the subpath of $P_0$ with endpoints $v$ and $w$. We examine two cases: 

\begin{enumerate}
\item At least $k+1$ paths of ${\cal W}$ (i.e. rails of the railed annulus) meet $C$. 
Then, as $|{\cal P}({\cal L})|=k,$ there exist two rails $W_i, W_j\in {\cal W}$ and a path $P\in{\cal P}({\cal L})$ such that 
$m_{\cal L}(W_i)=m_{\cal L}(W_j)=f_{\cal L}(P)$. 
Let $V(C_1)\cap V(W_i)$ be the vertices of path $Q_{1,i}$ and $V(C_1)\cap V(W_j)$ the vertices of path $Q_{1,j}$. 
Then, we let $x\in V(C_1)$ be the endpoint of $Q_{1,i}$ that is not $w'_i$ and $y\in V(C_1)$ be the endpoint of $Q_{1,j}$ that is not $w'_j$ (notice that $x$ and $y$ can coincide with $u$ and $v$). 
Let also $x'$ be the vertex of $V(P)\cap V(W_i)$ with the least index in the natural ordering of $W_i$  and $y'$ be the vertex of $V(P)\cap V(W_j)$ with the least index in the natural ordering of $W_j$. 
We observe that there exist two vertex disjoint paths $P_1$ and $P_2$ with endpoints either $v,x$ and $u,y$ or $v,y$ and $u,x,$ respectively. 
We define path $P_{u,v}=(C\setminus P(x',y')) \cup W_{i}(x,x') \cup W_{j}(y,y') \cup P_1\cup P_2$. Path $P_{u,v}$ has the desired properties. See also Figure~\ref{lemma11_1}.

\item There exist $k'= k+1$ paths, say ${\cal W}'=\{W_1, \ldots, W_{k'}\},$  of ${\cal W}$ that do not meet $C$. As the penetration of $C$ is at least $k+1,$ 
for every $j\in \{r-k, \ldots, r\},$ $V(C_j\cap C)\neq \emptyset$. For every $i\in \{1, \ldots, k'\}$ and every $j\in \{r-k, \ldots, r\}$ we assign to the vertex $w_i^j$ of $V(W_i\cap C_j)$ with the least index in the natural ordering of $W_i,$ a ``color" from the set 
$\{1, \ldots, k\}$ as follows: $c_{\cal L}(w_i^j)=c_{\cal L}(v)$
if there exists a $v\in V(C)$ and a subpath $C_j(w_i^j,v)$ (starting from $w_i^j$ and following $C_j$ in counter-clockwise order) such that it does not contain any other vertices of $V(C)$ as internal vertices.
For every $W_i\in {\cal W}',$ we assign to $W_i$ a set of colors, ${\chi}_i=\bigcup_{j=1}^{k+1} c_{\cal L}(w_i^j)$. Let ${\cal P}$ be the set of all maximal paths of $C_r$ without internal vertices in $C$. Certainly, any $W_i\in {\cal W}'$ intersects exactly one path of ${\cal P}$. We define the equivalence relation $\sim$ on the set of rails ${\cal W}'$ as follows: $W_i\sim W_{l}$ if and only if $W_i$ and $W_{l}$ intersect the same path of ${\cal P}$. We distinguish two subcases:

\begin{itemize}

\item The number of equivalence classes of $\sim$ is $k'$. Then, there exist two rails $W_i,W_{l} \in {\cal W}'$ and $j_i,j_l\in \{r-k, \ldots, r\}$ such that $c_{\cal L}(w_i^{j_i})=c_{\cal L}(w_{l}^{j_i})=c_{\cal L}(P)$ for some path $P\in {\cal P}({\cal L})$.

\item The number of equivalence classes of $\sim$ is strictly less than $k'$. Then, there exist two rails $W_i,W_{l} \in {\cal W}'$ such that $c_{\cal L}(w_i^j)=c_{\cal L}(w_{l}^j)$ for every $j\in \{r-k, \ldots, r\}$. Therefore, there exist $j_i,j_l\in \{r-k, \ldots, r\}$ with $j_i\neq j_l$ such that $c_{\cal L}(w_i^{j_i})=c_{\cal L}(w_{l}^{j_l})=c_{\cal L}(P)$ for some path $P\in {\cal P}({\cal L})$ (this holds because $ |\{r-k, \ldots, r\}|=k+1$ -- see also Figure~\ref{lemma11_2}).\end{itemize}

For both subcases, as $c_{\cal L}(w_i^{j_i})=c_{\cal L}(P),$ there exist a $v_j\in V(P)$ and a subpath $C_j(w_i^{j_i},v_j)$ of $C_j$ and, similarly, as $c_{\cal L}(w_{l}^{j_l})=c_{\cal L}(P),$ there exist a $v_{j_l}\in V(P)$ and a subpath $C_j(w_{l}^{j_l},v_{j_l})$ of $C_j$. These two subpaths do not contain any other vertices of $C$ apart from $v_{j_i}$ and $v_{j_l},$ respectively. Moreover, let $x$ be the vertex of $V(W_i\cap C_1)$ of the least index in the natural ordering of $W_i$ and $y$ the vertex of $V(W_{l}\cap C_1)$ of the least index in the natural ordering of $W_{l}$. As in case 1, observe that there exist two vertex disjoint paths $P_1$ and $P_2$ with endpoints either $v,x$ and $u,y$ or $v,y$ and $u,x,$ respectively. We define path $P_{u,v}=(C\setminus P(v_{j_i},v_{j_l}))\cup C_j(w_i^{j_i},v_{j_i}) \cup C_j(w_{l}^{j_l},v_{j_l}) \cup W_i(w_i^{j_i},x) \cup W_{l}(w_{l}^{j_l},y) \cup P_1 \cup P_2$. Path $P_{u,v}$ has the desired properties

\end{enumerate}
\end{proof}

\begin{figure}[ht]
    \centering
    \includegraphics[width=.7\textwidth]{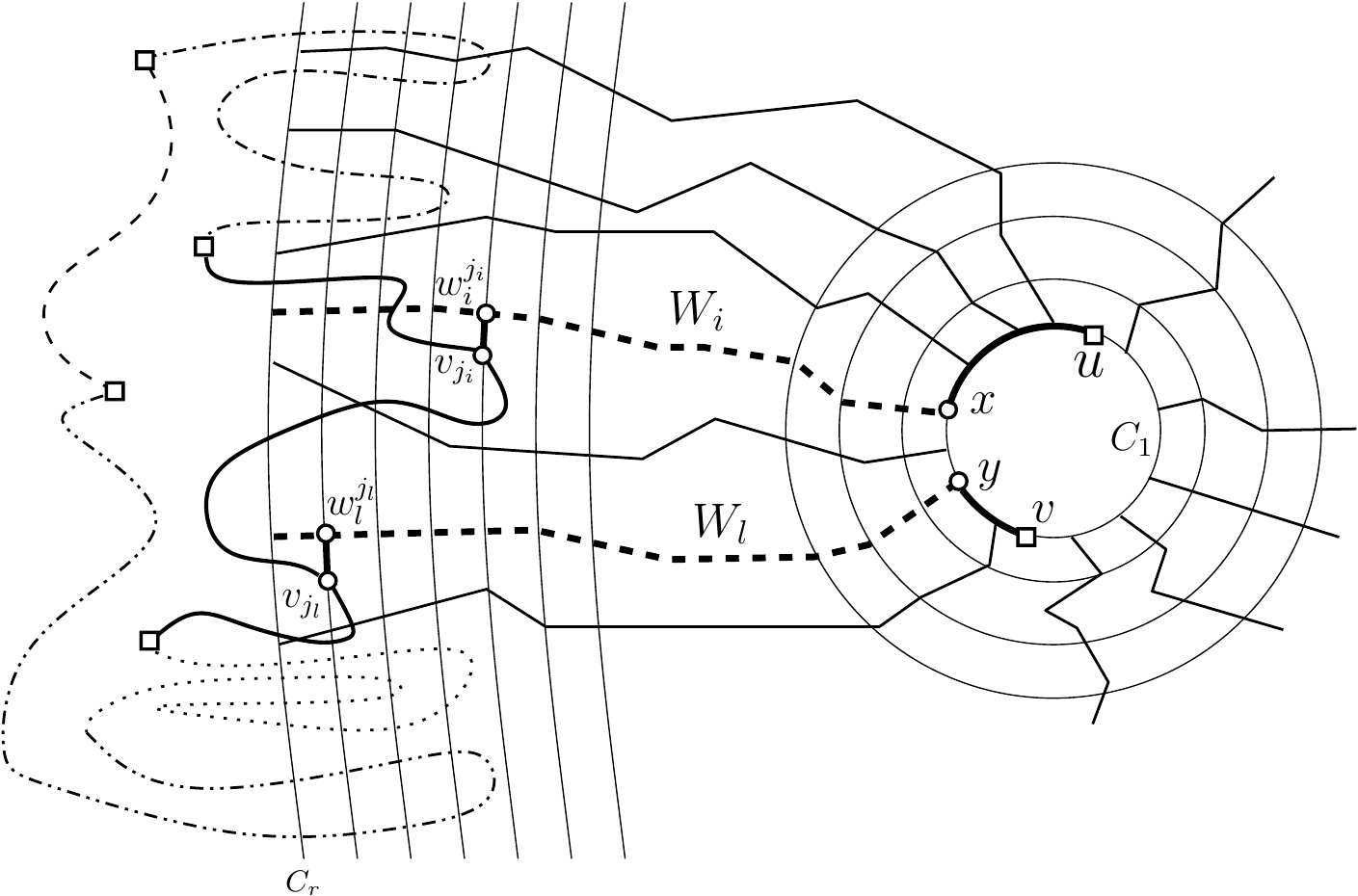}
    \caption{Visualization of proof of Lemma~\ref{lem:rails}, case 2, subcase 2.}
    \label{lemma11_2}
\end{figure}

\begin{proof}[Proof of Lemma~\ref{lem:color_irrelevant}] 
We first prove that if $(\hat{C}_{i+b},R_i,k)$ is a \yes-instance of {\sc PAC} for every $i\in\{1, \ldots, r-b\},$ then $(G,R,k)$ is a \yes-instance of {\sc PAC} iff $(G,R\setminus v,k)$ is a \yes-instance of {\sc PAC}.\\
For the non-trivial direction, we assume that $(G,R\setminus v,k)$ is a \yes-instance of {\sc PAC} and we have to prove that $(G,R,k)$ is also a \yes-instance of {\sc PAC}. Let $S\subseteq R$ with $|S|\leq k$. We have to prove that $S$ is cyclable in $G$. We examine two cases:
\begin{enumerate}

\item $v\notin S$. As $(G,R\setminus v,k)$ is a \yes-instance of ${\rm \Pi},$ clearly there exists a cyclic linkage ${\cal L}=(C,S)$ in $G,$ i.e., $S$ is cyclable in $G$. 

\item $v\in S$. As $r\geq k(98k+1)$ and $S\leq k,$ there exists $i$ such that $A_{i,i+98k}\cap S=\emptyset$.
We distinguish two sub-cases:\medskip

\noindent{\em Subcase 1.} $S\subseteq \bar{C}_{i+98k+1}$. Then, as $(\hat{C}_{i+98k+1},R_{i+98k+1},k)$
is a \yes-instance of ${\rm \Pi},$ then $S$ is cyclable in $\hat{C}_{i+98k+1}$ and therefore also in $G$.
\medskip

\noindent{\em Subcase 2.} There is a partition $\{S_{1},S_{2}\}$ of $S$ into two non-empty sets, such that 
$S_{1}\subset \mathring{C_{i}}$ and $S_{1}\cap  \bar{C}_{i+98k+1}=\emptyset$.
As $R$ is $32k$-dense in ${\cal C},$ there exists a vertex $v_{1}\in S\cap A_{i+k+1,i+33k+1}$
and a vertex $v_{2}\in S\cap A_{50+k+1,i+82k+1}$. For $i\in\{1,2\},$ let $S_{i}'=S_{i}\cup\{v_{i}\}$ and observe that $|S_{i}|\leq k$. 
Let ${\cal C}_{1}=\{C_{i+49k},\ldots,C_{i}\}$ and ${\cal C}_{2}=\{C_{i+49k},\ldots,C_{98k}\}$.
As $(\hat{C}_{i+98k+1},R_{98k+1},k)$ is a \yes-instance of ${\rm \Pi},$ $S_{1}'$ is cyclable in $\hat{C}_{i+98k+1}$. Also,  $(G,R\setminus v,k)$ is a \yes-instance, $S_{2}'$ is cyclable in $G$. For each $i\in\{1,2\},$
 there exists a cyclic linkage ${\cal L}_i=(C_{i},S_{i}')$ that has penetration 
at least $k+1$ in ${\cal C}_{i}$. We may assume that ${\cal L}_{i}$ is ${\cal C}_{i}$-cheap. 
Then, By Lemma~\ref{lem:problem_irrelevant}, the penetration of ${\cal L}_{i}$ in ${\cal C}_{i}$ 
is at most $49k$. Let ${\cal L}_{i}'=(C_{i},S_{i}), i\in\{1,2\}$.
For notational convenience we rename ${\cal C}_{1}$ and ${\cal C}_{2}$ where
${\cal C}_{1}=\{C^{1}_{1},\ldots,C^{1}_{49k+1}\}$ and ${\cal C}_{2}=\{C^{2}_{1},\ldots,C^{2}_{49k+1}\}$ 
(notice that $C_{49k+1}^{1}=C_{1}^{2}$).
Let $x,y$ be two distinct vertices in $C_{i+49k}$. 
For $i\in\{1,2\},$ we apply Lemma~\ref{lem:rails},
for $r=49k+1,$ $k,$ ${\cal C}_{i},$ ${\cal W},$ and $x$ and $y$
and obtain two paths $P_{i}, i\in\{1,2\},$ such that $S_{i}\subseteq V(P_{i})$
and whose endpoints are $x$ and $y$. Clearly, $P_{1}\cup P_{2}$ is a cycle
whose vertex set contains $S$ as a subset. Therefore $S$ is cyclable in $G,$ as required (see Figure~\ref{lemma10}).
\end{enumerate}
\phantom{.}
\vspace{-12.06mm}
\end{proof}

\begin{figure}[ht]
    \centering
    \includegraphics[width=.7\textwidth]{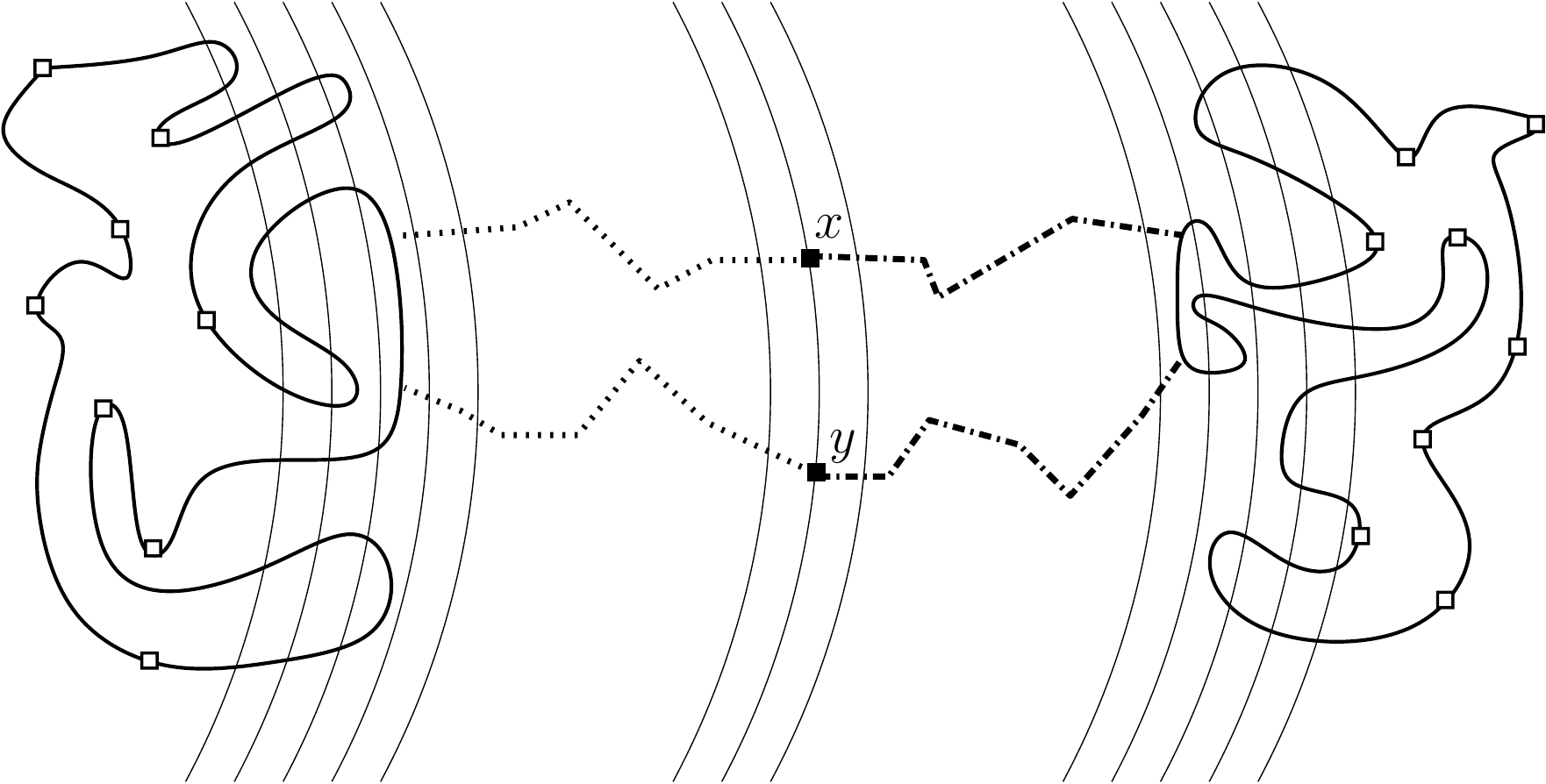}
    \caption{The squares of the right (resp. left) part represent the vertices of $S_1$ (resp. $S_2$). The connection between two cycles via rails and through $x$ and $y$ is derived from a double application of Lemma~\ref{lem:rails}. } 
    \label{lemma10}
\end{figure}

\section{Dynamic Programming for Planar Cyclability}
\label{DP}

\noindent In this section we present a dynamic programming algorithm for solving {\sc Cyclability} on graphs of bounded treewidth. We obtain the following algorithm.
\medskip

\noindent{\bf Algorithm DP}$(G,R,k,q,{\cal D})$\\
\noindent{\sl Input:} A graph $G,$ a vertex set $R\subseteq V(G),$ two non-negative integers $k$ and $q,$ where $k\leq q,$ and a tree decomposition ${\cal D}$  of $G$ of width $q$.\\
\noindent{\sl Output:} An answer whether $(G,R,k)$ is a \yes-instance of {\sc Planar Annotated  Cyclability} problem, or not.\\
\noindent{\sl Running time:} $2^{2^{O(q\cdot \log q)}}\cdot n$.
\medskip

\noindent We observe that the question of {\sc Planar Annotated  Cyclability} can be expressed in  monadic second-order logic (MSOL). It is sufficient to notice that 
an instance $(G,R,k)$ is a \yes-instance of {\sc Planar Annotated  Cyclability} if and only if for any (not necessarily distinct) $v_1,\ldots,v_k\in R,$ there are sets $X\subseteq V(G)$ and $S\subseteq E(G)$ such that $v_1,\ldots,v_k\in X$ and $C=(X,S)$ is a cycle. The property of  $C=(X,S)$ being a cycle is equivalent to asking whether
\begin{itemize}
\item[{\it i})] for any $x\in X,$ there are two distinct $e_1,e_2\in S$ such that $x$ is incident to $e_1$ and $e_2,$ 
\item[{\it ii})] for any $x\in X$ and any three pairwise distinct $e_1,e_2,e_3\in S,$ $e_1$ is not incident to $x$ or $ e_2$ is not incident to $x$  or 
$e_3$ is not incident to $x,$ and 
\item[{\it iii})] for any $Z_1,Z_2\subseteq X$ such 
that  $Z_1\cap Z_2=\emptyset,$ $Z_{1}\neq\emptyset,$ $Z_{2}\neq \emptyset$ and $Z_1\cup Z_2=X,$ 
there is $\{x,y\}\in S$ such that $x\in Z_1$ and $y\in Z_2$.
\end{itemize}
By the celebrated  Courcelle’s theorem (see, e.g., \cite{CourcelleE12,Courcelle90}), any 
problem that can be expressed in MSOL can be solved in linear time for 
graphs of bounded treewidth.\\
\noindent As we saw, {\sc Planar Annotated Cyclability}  can be solved in $f(q,k)\cdot n$ steps   if the treewidth of an input graph is at most $q$, for some computable function $f$.\\ 
As the general estimation of $f$ provided by  Courcelle’s theorem is immense, we give below a dynamic programming algorithm in order to achieve a  more reasonably running time.\medskip

\noindent First we introduce some notation.
\medskip

\noindent For every two integers $a$ and $b$, with $a<b$, we denote by $\intv{a}{b}$ the set of integers $\{a, a+1, \ldots, b\}$.
\noindent Let $S$ be a set and $i\in \Bbb{N}$. We define  $S^{[i]}=\{A\subseteq S \mid |A|= i\}$.

\paragraph{Sub-cyclic pairs.} Let $G$ be a graph, $C$ a cycle in $G$, and $\{A,X,B\}$ a partition of $V(G)$ such that no edge of $G$ has one endpoint in $A$ and the other in $B$. 
 The restriction of $C$ in $G[A \cup X]$ is called a {\em sub-cyclic} pair of $G$ (with respect to $A$, $X$ and $C$). We denote such a  sub-cyclic pair  by $({\cal Q}, Z)$, where ${\cal Q}$ contains the  connected components of the restriction of $C$ in $G[A \cup X]$ (observe that ${\cal Q}$ can contain isolated vertices, a unique cycle, and disjoint paths) and $Z=V(C)\cap X$.

\paragraph{Nice tree decompositions.} Let $G$ be a graph. A tree decomposition ${\cal D}=(T,{\cal X})$ of $G$ is called a {\em nice tree decomposition} of $G$ if $T$ is rooted to some leaf $r$ and:

\begin{enumerate}

\item for any leaf $l\in V(T)$ where $l\neq r$, $X_l=\emptyset$ (we call $X_l$ {\em leaf node} of ${\cal D}$, except from $X_r$ which we call {\em root} node)

\item the root and any non-leaf $t\in V(T)$ have one or two children

\item if $t$ has two children $t_1$ and $t_2$, then $X_t=X_{t_1}=X_{t_2}$ and $X_t$ is called {\em join} node

\item if $t$ has one child $t'$, then

\begin{itemize}

\item either $X_t=X_{t'}\cup \{v\}$ (we call $X_t$ {\em insert} node and $v$ {\em insert} vertex)

\item or $X_{t}=X_{t'}\setminus \{v\}$ (we call $X_t$ {\em forget} node and $v$ {\em forget} vertex).

\end{itemize}

\end{enumerate}

\paragraph{Pairings.} Let $W$ be a set. A {\em pairing} of $W$ is an undirected graph $H$ with vertex set $V(H)\subseteq W$ and where each vertex has degree at most 2 (a loop contributes 2 to the degree of its vertex) and if $H$ contains a cycle then this cycle is unique and all vertices not in this cycle have degree 0. Moreover, $H$ may also contain the vertex-less loop. 
We denote by ${\cal P}(W)$ the set of all pairings of $W$. It is known that if $|W|=w$ then $|{\cal P}(W)|=2^{O(w\cdot \log w)}$. 

\paragraph{Edge lifts.} Let $G$ be a graph and $v\in V(G)$ such that $\deg_G(v)=2$. Let also  $N_G(v)=\{u,w\}$. We say that the operation of deleting edges $\{v,u\}$ and $\{v,w\}$ and adding edge $\{u,w\}$ (if it does not exist, i.e. we do not allow double edges) is the {\em edge lift} from vertex $v$. We denote by {\sf lift}$(G,v)$ the graph resulting from $G$ after the edge lift from $v$.\\
For a vertex set $L\subseteq V(G)$ and a vertex $v\in V(G)$ we say that graph $H$={\sf lift}$(G,v)$ is the result of an $L${\em -edge-lift} if $v\in L$.

\medskip\medskip

\noindent Let $(G,R,k)$ be an instance of $p$-{\sc  Annotated Cyclability}. 
 Let also ${\cal D}=(T,{\cal X}, r)$ be a nice tree decomposition of $G$ of width $w$, where $r$ is the root of $T$. For every $x\in V(T)$ let $T_t$ be the subtree of $T$ rooted at $t$ (the vertices of $T_t$ are $t$ and its descendants in $T$). Then for every $t\in V(T)$, we define
$$G_t=G\Bigl{[}\bigcup_{t'\in V(T_t)}X_{t'}\Bigr{]} \mbox{~and~} V_{t}=V(G_{t}).$$
For every $i\in\Bbb{Z}_{>0}$, we  set ${\cal R}_t^i=(V(G_t)\cap R)^{[i]}$. 
We also denote ${\cal R}_{t}=\bigcup_{i=1}^{k}{\cal R}_{t}^{i}$.
\medskip

If  $({\cal Q},Z)$ is a  sub-cyclic pair  of $G_{t}$ where $X_t$ is thought of as the separator and $Z\subseteq X_{t}$,
we simply say that  $({\cal Q},Z)$ is a  {\em  sub-cyclic pair on $t$}.
Notice that each sub-cyclic pair $({\cal Q},Z)$ on $t$
corresponds to a  pairing in ${\cal P}(X_t)$, which we denote by  $P_{{\cal Q},Z}$ (just dissolve all vertices of ${\cal Q}$ that do not belong to $X_{t}$).

Let $P$ be a pairing of $X_t$ and $S$ be a subset of $V(G_t)$. We say that vertex set $S$ {\em realizes}  $P$ in $G_t$ if there exists a sub-cyclic pair $({\cal Q},Z)$ on $t$
such that $P_{{\cal Q},Z}=P$ and $S \subseteq V(\cupall {\cal Q})$.

We also define the {\em signature} of $S$ in $G_t$ to be the set of all pairings of $X_t$ that $S$ realizes and we denote it by $\sig_t(S).$ Notice that $\sig_t(S)\subseteq {\cal P}(X_t)$, therefore $|\sig_t(S)|=2^{O(w\cdot \log w)}$.
\vskip0.5cm

\paragraph{Tables.} We describe the tables of the dynamic programming algorithm.
For each $t\in V(T)$, we define  ${\cal C}_{t}=\intv{0}{k}\times X_{t}^{[i]}\times {\cal P}(X_t)$ and
 \begin{eqnarray*}
 {\cal F}_t & = & \{(i,K,{\cal P})\in{\cal C}_{t}\mid  \exists S\in {\cal R}_t^{i} \text{ such that } K=X_{t}\cap S  \mbox{ and } \sig_{t}(S)={\cal P} \}
 \end{eqnarray*}
We call ${\cal F}_t$ the {\em table} at node $t\in V(T)$. As $|{\cal P}(X_t)|={2^{O(w\cdot \log w)}}$, it follows that $|{\cal F}(t)|=2^{2^{O(w\cdot \log w)}}$. 
\medskip

\noindent Observe that $(G,R,k)$ is a \yes-instance of $p$-{\sc Annotated Cyclability} if and only if ${\cal F}_r=\{(0,\emptyset,P_r), (1, \emptyset,P_r), \ldots, (k,\emptyset,P_r)\}$, where $P_r$ is the unique pairing of ${\cal P}(X_r)$, i.e.,  the pairing that is the vertex-less loop (i.e., 
contains no vertices and a single edge with no endpoints).

\paragraph{New pairings from old.}
Before we describe the dynamic programming algorithm we need some definitions.
Suppose that $t$ is an insert node of ${\cal D}$ and $X_t=X_s\cup \{v\}$, where $s$ is the only child of $t$ 
in $T$ and $v\in V(G)$. Let $E_v^t = \{\{v,u\}\in E \mid u\in X_t\}$.
We denote by ${\cal P}_v^{\rm aux}$ the set of all graphs $(V, E)$ where $V\subseteq N_{G_{t}}(v)\cup \{v\}$ and $E \subseteq E_v^t $. For any $P \in {\cal P}(X_s)$, $\widetilde{P}\in {\cal P}_v^{\rm aux}$, and $L\subseteq X_{t}$, we define:
$$P \oplus_{L} \widetilde{P}= \{ P' \in {\cal P}(X_t) \mid 
P \text{ results from } P \cup \widetilde{P} \text{ after a sequence of $L$-edge-lifts} \}.$$  
\medskip
\begin{figure}[ht]
    \centering
    \includegraphics[width=.9 \textwidth]{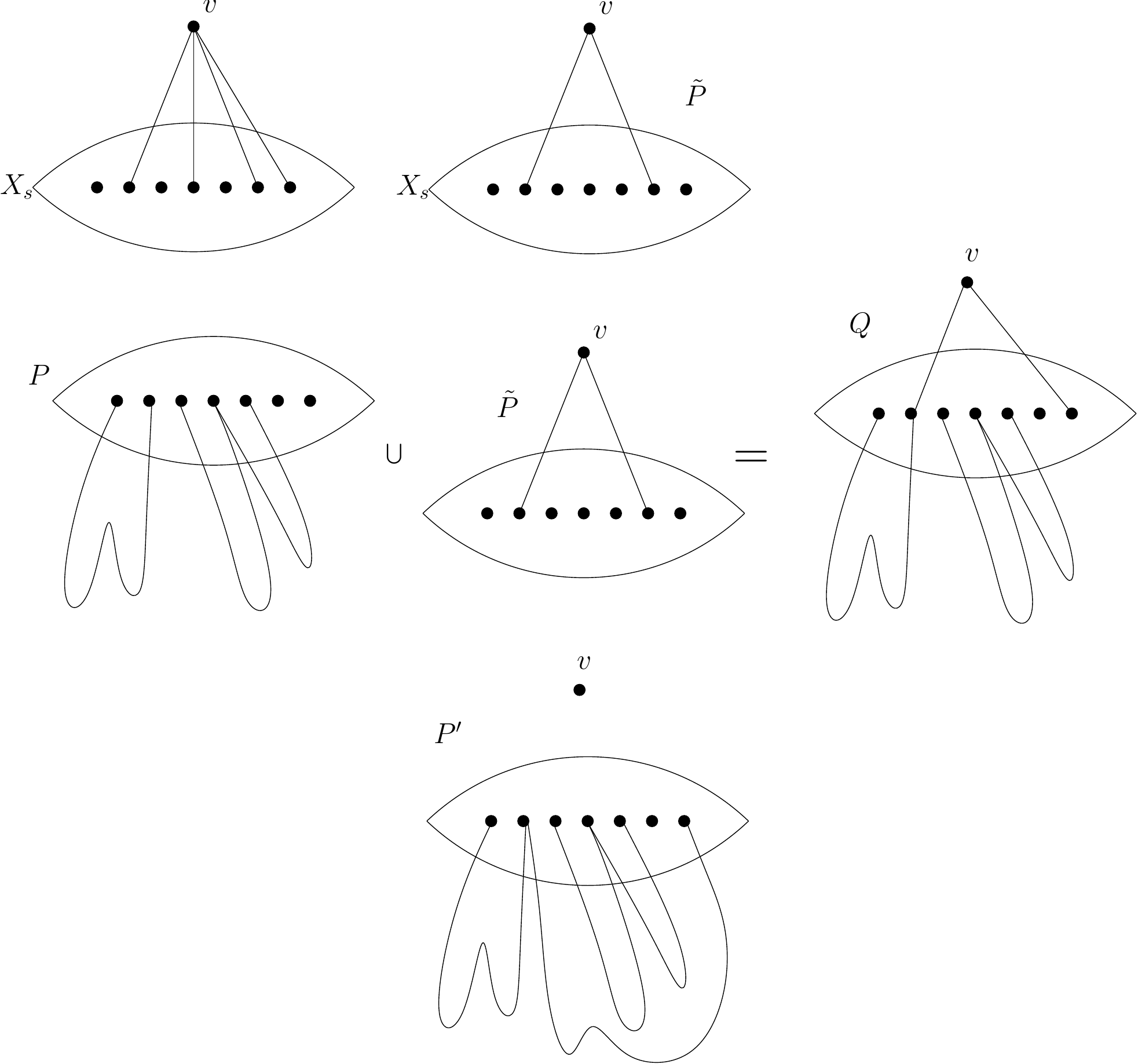}
        \caption{At the top we depict the neighborhood of node $v$ in $X_s$ (at the left) and an element, $\tilde{P}$ of $P_{v}^{aux}$ at the right. In the middle we depict the result, $Q$, of the union $P \cup \tilde{P}$, where $P \in {\cal P}(X_s)$. At the bottom we have the result, $P' \in {\cal P}(X_s \cup \{v\}) = {\cal P}(X_t)$, of lifting $v$ in $Q$.}
    \label{waslsl}
\end{figure}
See Figure~\ref{waslsl} for a visualization of the above definitions.
For every $P'\in {\cal P}(X_t)$ and $L\subseteq X_{t}$, we define

$$ \zeta_{L}(P')=\{ P \in {\cal P}(X_s) \mid \exists \widetilde{P}\in {\cal P}_v^{\rm aux} \text{ such that } P' \in {P \oplus_{L} \widetilde{P}} \}$$

\medskip

\noindent We are now ready to describe the  dynamic programming algorithm.
We distinguish the following  cases for the computation of ${\bf table}(t)$, $t\in V(G)$:

\paragraph{Node $t$ is a {\bf leaf node}:}\! as $X_{t}=\emptyset$,  we have that ${\cal F}(t)=\{(0,\emptyset,G_{\emptyset})\}$ where $G_{\emptyset}$ is the void graph.

\paragraph{Node $t$ is an {\bf insert node}:} Let $X_t=X_s\cup \{v\}$, where $s$ is the unique child of $t$ in $T$. 
We construct $\tab(t)$ by using the following procedure:

\begin{tabbing}
Procedure {\bf make\_join}\\
{\sl Input:} a subset ${\cal A}$ of ${\cal C}_{s}$\\
{\sl Output:} a subset ${\cal B}$ of ${\cal C}_{t}$\\
let ${\cal B}=\emptyset$\\
for \=  $(i,K,{\cal P})\in {\cal A}$ \\
    \> if \= $v\in R$ and $i<k$ then 	\\
    \> \> let \=${\cal B}={\cal B}\cup\{(i+1,K\cup\{v\},{\cal P}')\}$\\
    \> \> \> where ${\cal P}'=\{P\in {\cal P}(X_{t})\mid \zeta_{X_s\setminus K}(P)\cap {\cal P}\neq \emptyset\}$ \\
\> let \=${\cal B}={\cal B}\cup\{(i,K,{\cal P}'')\}$\\
 \> \> \> where ${\cal P}''=\{P\in {\cal P}(X_{t})\mid \zeta_{X_t\setminus K}(P) \cap {\cal P}\neq \emptyset\}$\\
 
\end{tabbing}

\begin{lemma}
\label{insert}
${\cal F}_{t}=\mbox{\bf make\_join}({\cal F}_{s})$.
\end{lemma}

\begin{proof} We first prove that $\mbox{\bf make\_join}({\cal F}_{s})\subseteq {\cal F}_t$. 
Let $(i+1,K \cup \{v\},{\cal P})\in \mbox{\bf make\_join}({\cal F}_{s})$ with $v \in R$ and $i<k$ (the other case is similar).
We prove that $(i+1,K \cup \{v\},{\cal P})\in {\cal F}_t$. \\
By the operation of the procedure {\bf make\_join} we have that there exists a triple $(i, K, {\cal P}) \in {\cal F}_s$ such that ${\cal P}'=\{P\in {\cal P}(X_{t})\mid \zeta_{X_s\setminus K}(P)\cap {\cal P}\neq \emptyset\}$. 
Let $S \subseteq {\cal R}_s^i$ be the annotated vertex set which justifies the existence of $(i,K,{\cal P})$ in ${\cal F}_s$, i.e. $X_s \cap S = K$ and $\sig_s(S) = {\cal P}$.
Now, let $S'=S \cup \{v\}$. Clearly, $S' \subseteq {\cal R}_t^{i+1}$ (where $i+1 \leq k$) and $X_t \cap S' = K \cup \{v\}$.\\
It remains to show that $\sig_t(S') = {\cal P'}$ or, equivalently,
 $\forall P\in {\cal P}(X_t)$ it holds that $P \in \sig_t(S') \Leftrightarrow \zeta_{X_s \setminus K}(P) \cap {\cal P} \neq \emptyset$. 
Let $P \in \sig_t(S')$. We distinguish three cases:
\begin{itemize}
\item {\bf Case 1:} $\deg_P(v) = 0$.
 Then, $P^* = P \setminus \{v\} \in X_s$ and $P = P^* \cup (\{v\}, \emptyset)$ (notice that $(\{v\}, \emptyset) \in {\cal P}_v^{aux}$), 
which means that $P^* \in \zeta_{X_s \setminus K}(P)$. 
It is not hard to confirm that $P^* \in {\cal P}$ because $S$ realizes $P^*$ in $G_s$. It follows that $P \in {\cal P}'$.
\item {\bf Case 2:} $\deg_P(v) = 1$. 
Let $u$ be the only neighbor of $v$ in $P$. 
Then, $P^* = P \setminus \{v\} \in X_s$ and $P = P^* \cup (\{v\}, \{v,u\})$, which means that $P^* \in \zeta_{X_s \setminus K}(P)$. 
Again, $P^* \in {\cal P}$ because $S$ realizes $P^*$ in $G_s$, thus $P \in {\cal P}'$.
\item {\bf Case 3:} $\deg_P(v) = 2$. 
Let $N_P(v)=\{u,w\}$.
Then,  $P^* = P \setminus \{v\} \in X_s$ and $P = P^* \cup (\{v\}, \{\{v,w\}\{v,u\}\})$,
which means that $P^* \in \zeta_{X_s \setminus K}(P)$. 
As before, $S$ realizes $P^*$ in $G_s$, thus $P \in {\cal P}'$.
\end{itemize}
We have showed that $\sig_t(S') \subseteq {\cal P}'$. The converse, ${\cal P}' \subseteq \sig_t(S')$, is clear from the definition of ${\cal P}'$.\\
\noindent To conclude the proof, we have to show that ${\cal F}_{t} \subseteq \mbox{\bf make\_join}({\cal F}_{s})$. 
Let $(i, K, {\cal P})\in {\cal F}_t$. From the definition of ${\cal F}_t$, there exists a vertex set $S \subseteq {\cal R}_t^i$ that realizes every pairing of ${\cal P}$ and $X_t \cap S = K$. Let $P\in {\cal P}$ and assume that $v \notin R$. We consider three cases and the arguments are similar to the previous ones:
\begin{itemize}
\item {\bf Case 1:} $\deg_P(v) = 0$. Then, $P^* = P \setminus \{v\} \in X_s$ and $P = P^* \cup (\{v\}, \emptyset)$, 
which means that $P^* \in \zeta_{X_s \setminus K}(P)$. 
\item {\bf Case 2:} $\deg_P(v) = 1$. 
Let $u$ be the only neighbor of $v$ in $P$. 
Then, $P^* = P \setminus \{v\} \in X_s$ and $P = P^* \cup (\{v\}, \{v,u\})$, which means that $P^* \in \zeta_{X_s \setminus K}(P)$. 
\item {\bf Case 3:} $\deg_P(v) = 2$. 
Let $N_P(v)=\{u,w\}$.
Then,  $P^* = P \setminus \{v\} \in X_s$ and $P = P^* \cup (\{v\}, \{\{v,w\}\{v,u\}\})$,
which means that $P^* \in \zeta_{X_s \setminus K}(P)$. 
\end{itemize}

\noindent Let ${\cal P}^*=\{P^* \in {\cal P}(X_s) \mid P \in {\cal P}\}$. Clearly, for $S^*=S\setminus \{v\} \subseteq {\cal R}_t^i$ and $K^*=X_s \cap S^*$, we have that $\sig(S^*)={\cal P}^*$ and thus $(i-1, K^*, {\cal P}^*)\in {\cal F}_s$.

\noindent The case where $v\in R$ is similar. We conclude that ${\cal F}_{t} \subseteq \mbox{\bf make\_join}({\cal F}_{s})$, which completes the proof. 
\end{proof}

\paragraph{Node $t$ is a {\bf forget node}:} Let $X_t=X_s\setminus \{v\}$, where $s$ is the unique child of $t$ in $T$. 
Then 
$${\cal F}_t=\{(i,K\setminus \{v\},{\cal P})\mid \exists (i,K,{\cal P}')\in {\cal F}_s: \forall P\in {\cal P}(X_t),\ P\in {\cal P}\Leftrightarrow {\sf lift}(P,v)\in {\cal P}'\}$$

The proof that the right part of the above equality is ${\cal F}_t$, is similar to the one of Lemma ~\ref{insert}.

\paragraph{Node $t$ is a {\bf join node}:} Let $s_1$ and $s_2$ be the children of $t$ in $T$. Thus, $X_t = X_{s_1} = X_{s_2}$ and clearly ${\cal P}(X_t)={\cal P}(X_{s_1})={\cal P}(X_{s_2})$. Given a pairing $P\in {\cal P}(X_{t})$, we define 
$$\xi(P)=\{(P_{1},P_{2})\in {\cal P}(X_{t})\times {\cal P}(X_{t})\mid P_{1}\cup P_{2}=P\}$$
Then, ${\cal F}_t$  can be derived from ${\cal F}_{s_1}$ and ${\cal F}_{s_2}$ as follows:
\begin{eqnarray*}
{\cal F}_t & = & \{ (i,K,{\cal P})\mid \exists (i_{1},K_{1},{\cal P}_{1})\in {\cal F}_{s_1}\ \exists (i_{2},K_{2},{\cal P}_{2})\in {\cal F}_{s_2}:\\
&  & \ \ \ \ \ \ \ \ \ \ \  \ \ \  \ i= i_{1}+i_{2}-|K_{1}\cap K_{2}|, \\
&  & \ \ \ \ \ \ \ \ \ \ \  \ \ \  \ K_{1}\cup K_2=K\\
&  & \ \ \ \ \ \ \ \ \ \ \  \ \ \  \ \forall P\in{\cal P}\ \exists (P_{1},P_{2})\in\xi(P): P_{1}\in {\cal P}_{1},P_{2}\in {\cal P}_{2}\}.
\end{eqnarray*}

\begin{lemma} 
In the case where $t$ is a join node with children $s_1$ and $s_2$, ${\cal F}_t$ is computed as described above, given ${\cal F}_{s_1}$ and ${\cal F}_{s_2}$.
\end{lemma}

\begin{proof} Let ${\cal U}_t = \{(i,K,{\cal P}) \mid  \exists (i_{1},K_{1},{\cal P}_{1})\in {\cal F}_{s_1} \quad \exists (i_{2},K_{2},{\cal P}_{2})\in {\cal F}_{s_2}:
 i= i_{1}+i_{2}-|K_{1}\cap K_{2}|, \quad
K_{1}\cup K_2=K
\text{ and } \forall P\in{\cal P} \quad \exists (P_{1},P_{2})\in\xi(P): P_{1}\in {\cal P}_{1},P_{2}\in {\cal P}_{2}\}$.\\
We will only prove the nontrivial direction: ${\cal F}_t \subseteq {\cal U}_t$. Let $(i,K,{\cal P}) \in {\cal F}_t$. From the definition of ${\cal F}_t$, there exists a vertex set $S \subseteq {\cal R}_t^i$ that realizes every pairing of ${\cal P}$ and $X_t \cap S = K$. Let $P$ be any pairing of ${\cal P}$. Then, there exists a sub-cyclic pair $({\cal Q},Z)$ on $t$ that corresponds to pairing $P$. The restriction of $({\cal Q},Z)$ in $G_{s_1}$ (resp. $G_{s_2}$) is a sub-cyclic pair $({\cal Q}_1,Z_1)$ on $s_1$ (resp. $({\cal Q}_2,Z_2)$ on $s_2$) and clearly $Z_1 \subseteq X_{s_1}$ (resp. $Z_2 \subseteq X_{s_2}$. These sub-cyclic pairs meet some subsets $S_1$ and $S_2$ of $S$ respectively and correspond to parings $P_1\in {\cal P}_1$ and $P_2\in {\cal P}_2$.\\
Let $|S_1|=i_1$ and $|S_2|=i_2$. It is now easy to confirm that $i=i_1+i_2-|K_1 \cap K_2|$, $K_1 \cup K_2 = Z_1 \cup Z_2= K$ and that $(P_1,P_2)\in \xi(P)$.\\
As $P\in {\cal P}$ was chosen arbitrarily we conclude that $(i,K,{\cal P})\in {\cal U}_t$ and we are done. 
\end{proof}
\medskip

The dynamic programming algorithm that we described runs in $2^{2^{O(w\cdot \log w)}}\cdot n$ steps (where $w$ is the width of the tree decomposition) and solves {\sc Cyclability}.

\section{Hardness of the Cyclability Problem}
\label{hardness}
In this section, we examine the hardness of {\sc Cyclability}.

\subsection{Hardness for general graphs}
First, we show that it is unlikely that  {\sc Cyclability} is {\sf FPT} by 
proving Theorem~\ref{thm:W} (mentioned in the introduction).
For this, we first introduce some further notation.

A \emph{matching} is a set of pairwise non-adjacent edges. A vertex $v$ is \emph{saturated} in a matching $M$ if $v$ is incident to an edge of $M$.
By $x_1\ldots x_p$ we denote the path  with the vertices $x_1,\ldots,x_p$ and the edges $\{x_1,x_2\},\ldots,\{x_{p-1},x_p\}$, and we use 
$x_1\ldots x_px_1$ to denote the cycle with the vertices $x_1,\ldots,x_p$ and the edges $\{x_1,x_2\},\ldots,\{x_{p-1},x_p\},\{x_p,x_1\}$.
For a path $P=x_1\ldots x_p$ and a vertex $y,$ $yP$ ($Py$ resp.) is the path $yx_1\ldots x_p$ ($x_1\ldots x_py$ resp.).
If $P_1=x_1\ldots x_p$ and $P_2=y_1\ldots y_q$ are paths such that $V(P_1)\cap V(P_2)=\{x_p\}=\{y_1\},$ then
$P_1+P_2$ is the \emph{concatenation} of $P_1$ an $P_2,$ i.e., the path $x_1\ldots x_{p-1}y_1\ldots y_q$.
\medskip

\noindent We need some auxiliary results. The following lemma is due to Erd{\H{o}}s~\cite{Erdos62}. Define the function $f(n,\delta)$ by
$$f(n,\delta)=
\begin{cases}
\binom{n-\delta}{2}+\delta^2&\mbox{if } n\geq 6\delta-2,\\
\binom{(n+1)/2}{2}+(\frac{n-1}{2})^2&\mbox{if } n\leq 6\delta-3 \mbox{ and } n \mbox { is odd},\\
\binom{(n+2)/2}{2}+(\frac{n-2}{2})^2&\mbox{if } n\leq 6\delta-4 \mbox{ and } n \mbox { is even}.
\end{cases}
$$

\begin{lemma}[\cite{Erdos62}]\label{lem:erdos}
Let $G$ be a graph with $n\geq 3$ vertices. If $\delta(G)\geq n/2$ or $|E(G)|>f(n,\delta(G)),$ then $G$ is Hamiltonian.
 \end{lemma}

\begin{lemma}\label{lem:ham}
Let $k\geq 75$ be an odd integer and let $H$ be a graph such that 
\begin{itemize}
\item[i)] $(k-2)(k-3)/2<|E(H)|\leq k(k-1)/2+1,$
\item[ii)] $\delta(H)\geq (k-1)/2,$
\item[iii)] there is a set $S\subseteq E(H)$ such that $|S|>(k-2)(k-3)/2$ and $G[S]$ has at most $k+2$ vertices.
\end{itemize}
Then $H$ is Hamiltonian.
\end{lemma}

\begin{proof}
Let $H$ be an $n$-vertex graph that satisfies the above three conditions.
Let $S\subseteq E(H)$ be a set such that $|S|>(k-2)(k-3)/2$ and $G[S]$ has at most $k+2$ vertices. Let also $U=V(H)\setminus V(G[S])$. Denote by $R$ the set of edges of $G$ incident to vertices of $U$. Since 
 $|S|>(k-2)(k-3)/2$ and $|E(H)|\leq k(k-1)/2+1,$ $|R|\leq 2k-3$. Because $\delta(H)\geq (k-1)/2,$ $|R|\geq |U|\delta(H)/2\geq |U|(k-1)/4$. We have that $|U|\leq 7,$ i.e., $H$ has at most $k+9$ vertices. Then because $k\geq 75,$ we obtain that $n\geq 6\delta(G)-3,$  
$$\binom{(n+1)/2}{2}+\Big(\frac{n-1}{2}\Big)^2\leq\frac{(k-2)(k-3)}{2}<|E(H)|$$
and
$$\binom{(n+2)/2}{2}+\Big(\frac{n-2}{2}\Big)^2\leq\frac{(k-2)(k-3)}{2}<|E(H)|.$$
We have that $|E(H)|>f(n,\delta(H)),$ and 
by Lemma~\ref{lem:erdos}, $H$ is Hamiltonian.
\end{proof}

\noindent We are now in the position to prove Theorem \ref{thm:W}:

\begin{proof}[Proof of Theorem~\ref{thm:W}]
We reduce the {\sc Clique} problem.
Recall that {\sc Clique} asks for a graph $G$ and a positive integer $k,$ whether $G$ has a clique of size $k$. This problem is well known to be {\sf W[1]}-complete~\cite{DowneyF13} when parameterized by $k$. Notice that {\sc Clique} remains {\sf W[1]}-complete when restricted to the instances where $k$ is odd. To see it, it is sufficient to observe that if the graph $G'$ is obtained from a graph $G$ by adding a vertex adjacent to all the vertices of $G,$ then $G$ has a clique of size $k$ if and only if $G'$ has a clique of size $k+1$. Hence, any instance of {\sc Clique} can be reduced to the instance with an odd value of the parameter. Clearly, the problem is still {\sf W[1]}-hard if the parameter $k\geq c$ for any constant $c$.

Let $(G,k)$ be an instance of {\sc Clique} where $k\geq 75$ is odd. We construct the graph $G'_k$ as follows.
\begin{itemize}
\item For each vertex $x\in V(G),$ construct $s=(k-1)/2$ vertices $v_x^i$ for $i\in\{1,\ldots,s\}$ and form a clique 
of size $ns$ from all these vertices by joining them by edges pairwise.
\item Construct a vertex $w$ and edges $\{w,v_x^i\}$ for $x\in V(G),$ $i\in\{1,\ldots,s\}$.
\item For each edge $\{x,y\}\in E(G),$ construct the vertex $u_{xy}$ and the edges $\{u_{xy},v_{x}^i\},$
$\{ u_{xy},v_y^i\}$ for $i\in\{1,\dots,s\}$; we assume that $u_{xy}=u_{yx}$.
 \end{itemize}
Let $k'=k(k-1)/2+1$. It is straightforward to see that $G'$ is a split graph.
We show that $G$ has a clique of size $k$ if and only if there are $k'$ vertices in $G'_k$ such that there is no cycle in $G'_k$ that contains these $k'$ vertices.

Suppose that $G$ has a clique $X$ of size $k$. Let $Y=\{u_{xy}\in V(G')|x,y\in X,x\neq y\}$ and $Z=Y\cup \{w\}$. Because $|X|=k,$ $|Z|=k(k-1)/2+1=k'$. Observe that $Y$ is an independent set in $G'_k$ and $|Y|=|N_{G'}(Y)|$. Hence, for any cycle $C$ in $G'_k$ such that $Y\subseteq V(C),$ $V(C)\subseteq Y\cup N_{G'_k}(Y)$. Because $w\notin Y\cup N_{G'_k}(Y),$ $w$ does not belong to any cycle that contains the vertices of $Y$. We have that no cycle in $G'_k$ contains $Z$ of size $k'$.

Now we show that if $G$ has no cliques of size $k,$ then for any $Z\subseteq V(G'_k)$ of size $k',$ there is a cycle $C$ in $G'_k$ such that $Z\subseteq V(C)$. We use the following claim.

\medskip
\noindent
{\bf Claim.} {\it 
Suppose that $G$ has no cliques of size $k$.  Then for any non-empty $Z\subseteq \{u_{xy}|x,y\in V(G)\}$ of size at most  $k(k-1)/2+1,$ there is a cycle $C$ in $G'_k$ such that $Z\subseteq V(C)\subseteq Z\cup N_G(Z)$ and
$C$ has an edge $\{v_x^i,v_y^j\}$ for some $x,y\in V(G)$ and $i,j\in\{1,\ldots,s\}$.
}

\begin{proof}[of Claim]
For a set  $Z\subseteq \{u_{xy}|x,y\in V(G)\},$ we denote by $S(Z)$ the set of edges $\{\{x,y\}\in E(G)|u_{xy}\in Z\},$ and $H(Z)=G[S(Z)]$.

If $Z=\{u_{xy}\},$ then the triangle $u_{xy}v_{x}^1v_x^2u_{xy}$ is a required cycle, and the claim holds. Let $r=|Z|\geq 2$ and assume inductively that the claim is fulfilled for smaller sets. 

Suppose that $H(Z)$ has a vertex $x$ with ${\rm deg}_{H(Z)}(x)\leq (k-3)/2$. Let $N_{H(Z)}(x)=\{y_1,\ldots,y_t\}$. Notice that $t\leq (k-3)/2=s-1$.
 Denote by $Z'$ the set obtained from $Z$ by the deletion of $u_{xy_1},\ldots,u_{xy_t},$ and let $H'=H(Z')$. If $Z'=\emptyset,$ then the cycle $C=v_x^1u_{xy_1}v_x^2\ldots v_x^tu_{xy_t}v_x^{t+1}v_x^1$ satisfies the conditions and the claim holds. Suppose that $Z'\neq \emptyset$. Then, by induction, there is a cycle $C'$ in $G'_k$ such that $Z\subseteq V(C')\subseteq Z\cup N_G(Z)$ and $C'$ has an edge $\{v_a^i,v_b^j\}$ for some $a,b\in V(G)$ and $i,j\in\{1,\ldots,s\}$. We consider the path 
$P=v_x^1u_{xy_1}v_x^2\ldots v_x^tu_{xy_t}v_x^{t+1}$. Then we delete $\{v_a^i,v_b^j\}$ and replace it by the path $v_a^iPv_b^j$. Denote the obtained cycle by $C$. It is straightforward to verify that 
 $Z\subseteq V(C)\subseteq Z\cup N_G(Z)$ and $\{v_a^i,v_x^1\}\in E(C),$ i.e., the claim is fulfilled.

From now we assume that $\delta(H(Z))\geq (k-1)/2$. We consider three cases.

\medskip
\noindent{\bf Case 1.} $r\leq (k-2)(k-3)/2$.  

Consider the graph $G_{k-2}'$. We show that this graph has a matching $M$ of size $r$ such that every vertex of $Z$ is saturated in $M$. By the Hall's theorem (see, e.g., \cite{Diestel10}), it is sufficient to show that for any $Z'\subseteq Z,$ $|Z'|\leq |N_{G_{k-2}'}(Z')|$. Let $p$ be the smallest positive integer such that $|Z'|\leq p(p-1)/2$. 
By the definition of $G_{k-2}',$ $|N_{G_{k-2}'}(Z')|\geq p(k-3)/2$. Because $p\leq k-2,$  $|Z'|\leq p(p-1)/2 \leq p(k-3)/2\leq  |N_{G_{k-2}'}(Z')|$.

Let $M$ be a matching in $G_{k-2}\rq{}$ of size $r$ such that  every vertex of $Z$ is saturated in $M$. Clearly, $M$ is a matching in $G_k\rq{}$ that saturates $Z$ as well.
Let $x_1,\ldots,x_q$ be the vertices of $G$ such that for $i\in\{1,\ldots,q\},$ $\{v_{x_i}^1,\ldots,v_{x_i}^s\}$ contains saturated in $M$ vertices. Because $v_{x_i}^1,\ldots,v_{x_i}^s$ have the same neighborhoods, we assume without loss of generality that for  $i\in\{1,\ldots,q\},$ $v_{x_i}^1,\ldots,v_{x_i}^{t_i}$ are saturated. Observe that since $M$ is a matching in $G_{k-2}\rq{},$ $t_i\leq s-1$.  For $i\in\{1,\ldots,q\}$ and $j\in\{1,\ldots,t_i\},$ denote by $u_i^j$ the vertex of $Z$ such that $\{v_{x_i}^j,u_i^j\}\in M$. We define the path $P_i=v_{x_i}^1u_i^1v_{x_i}^2\ldots u_i^{t_i}v_{x_i}^{t_i+1}$ for $i\in\{1,\ldots,q\}$.   As all the vertices $v_{x_i}^j$ are pairwise adjacent, by adding the edges $\{v_{x_1}^{t_1+1},v_{x_2}^{1}\},\ldots, \{v_{x_{q-1}}^{t_{q-1}+1},v_{x_q}^{1}\},
\{v_{x_q}^{s_q+1},v_{x_1}^{1}\}$, we obtain from the the paths $P_1,\ldots,P_q$ a cycle. Denote it by $C$.  
We have that  $Z\subseteq V(C)\subseteq Z\cup N_G(Z)$ and $\{v_{x_1}^{t_1+1},v_{x_2}^{1}\} \in E(C),$ and we conclude that the claim holds.

\medskip
\noindent{\bf Case 2.} $(k-2)(k-3)/2<r$ and for any $S\subseteq E(H(Z))$ such that $|S|>(k-2)(k-3)/2,$ $H(Z)[S]$ has at least $k+3$ vertices.

We use the same approach as in Case 1 and show that $G_{k-2}'$ has a matching $M$ of size $r$ such that every vertex of $Z$ is saturated in $M$. We have to show that for 
any $Z'\subseteq Z,$ $|Z'|\leq |N_{G_{k-2}'}(Z')|$. If $|Z\rq{}|\leq (k-2)(k-3)/2,$ we use exactly the same arguments as in Case 1. Suppose that  $|Z\rq{}|> (k-2)(k-3)/2$.
Then $|S(Z\rq{})|=|Z'|> (k-2)(k-3)/2$. Hence, $H(Z)[S(Z\rq{})]$ has at least $k+3$ vertices. It implies that $|N_{G_{k-2}'}(Z')|\geq (k+3)(k-3)/2$. Because $k\geq 75$ and $|Z\rq{}|\leq r\leq k(k-1)/2+1,$
$|N_{G_{k-2}'}(Z')|\geq (k+3)(k-3)/2\geq  k(k-1)/2+1\geq |Z\rq{}|$. Given a matching $M$ that saturates $Z,$ we construct a cycle 
that contains $Z$ in exactly the same way as in Case 1 and prove that the claim holds.

\medskip
\noindent{\bf Case 3.} $(k-2)(k-3)/2<r$ and there is $S\subseteq E(H(Z))$ such that $|S|>(k-2)(k-3)/2$ and $H(Z)[S]$ has at most $k+2$ vertices. 

By Lemma~\ref{lem:ham}, $H(Z)$ is Hamiltonian. Let $p=|V(H(Z))|$ and denote by $R=x_1\ldots x_px_1$ a Hamiltonian cycle in $H(Z)$.  Let $U=\{u_{x_1x_2},\ldots,u_{x_{p-1}x_1}\}$ and let
$Z\rq{}=Z\setminus U$.

We again consider $G_{k-2}'$. We show that this graph has a matching $M$ of size $|Z'|$ such that every vertex of $Z'$ is saturated in $M$. We have to prove that  for 
any $Z''\subseteq Z',$ $|Z''|\leq |N_{G_{k-2}'}(Z'')|$. If $|Z''|\leq (k-2)(k-3)/2,$ we use exactly the same arguments as in Case 1. Suppose that  $|Z''|> (k-2)(k-3)/2$. Let $q$ be the smallest positive integer such that $|Z''|\leq q(q-1)/2$. Clearly, $q>k-2$.
We consider the following three cases depending on the value of $q$.

\medskip
\noindent{\bf Case a.} $q=k-1$. Then $H(Z'')$ has at least $k-1$ vertices and at least $(k-2)(k-3)/2+1$ edges.
Because $|Z|\leq k(k-1)/2+1,$  $H(Z)$ has  at most $2k-3$ edges that are not edges of $H(Z'')$.  Because $\delta(H(Z))\geq (k-1)/2$ and $k\geq 75,$ $H(Z)$ has at most 4 vertices that are not adjacent to the edges of $H(Z'')$.  Then at most 8 edges of the Hamiltonian cycle $R$ in $H(Z)$ do not join vertices of $H(Z'')$ with each other. We obtain that at least $k-9$ edges of $R$ join 
vertices of $H(Z'')$ with each other. 

Suppose that $H(Z'')$ has $k-1$ vertices. 
Then $|Z''|\leq (k-1)(k-2)/2-(k-9)\leq (k^2-5k+20)/2$. 
Because $H(Z'')$ has $k-1$ vertices, $|N_{G_{k-2}'}(Z'')|=(k-1)(k-3)/2$.
Since $k\geq 75,$ $|Z''|\leq |N_{G_{k-2}'}(Z'')|$.  

Suppose that $H(Z'')$ has $k$ vertices. If $H(Z)$ has a vertex $x$ that is not adjacent to the edges of $H(Z''),$ then at least $(k-1)/2$ vertices of $Z$ that correspond to the edges incident to $x$ are not in $Z''$. Then 
$|Z''|\leq |Z|-(k-1)/2-(k-9)\leq(k^2-4k+21)/2$. Because $|N_{G_{k-2}'}(Z'')|=k(k-3)/2$ and $k\geq 75,$ $|Z''|\leq |N_{G_{k-2}'}(Z'')|$. If $H(Z)$ has no vertex that is not adjacent to the edges of $H(Z''),$ then the edges of $R$ join vertices of $H(Z'')$ with each other.
We have that $|Z''|\leq k(k-1)/2-k=k(k-3)/2$ and $|Z''|\leq |N_{G_{k-2}'}(Z'')|$.

Finally, if $H(Z'')$ has at least $k+1$ vertices, then $|N_{G_{k-2}'}(Z'')|\geq (k+1)(k-3)/2\geq (k-1)(k-2)/2\geq |Z''|$.

\medskip
\noindent{\bf Case b.} $q=k$. Then $H(Z'')$ has at least $k$ vertices and at least $(k-1)(k-2)/2+1$ edges.
Because $|Z|\leq k(k-1)/2+1,$  $H(Z)$ has  at most $k-1$ edges that are not edges of $H(Z'')$.  
Because $\delta(H(Z))\geq (k-1)/2$ and $k\geq 75,$ $H(Z)$ has at most 2 vertices that are not adjacent to the edges of $H(Z'')$.  
Then at most 4 edges of the Hamiltonian cycle $R$ in $H(Z)$ do not join vertices of $H(Z'')$ with each other. We obtain that at least $k-4$ edges of $R$ join 
vertices of $H(Z'')$ with each other. 

Suppose that $H(Z'')$ has $k$ vertices. If $H(Z)$ has a vertex $x$ that is not adjacent to the edges of $H(Z''),$ then at least $(k-1)/2$ vertices of $Z$ that correspond to the edges incident to $x$ are not in $Z''$. Then 
$|Z''|\leq |Z|-(k-1)/2-(k-4)\leq(k^2-4k+11)/2$. Because $|N_{G_{k-2}'}(Z'')|=k(k-3)/2$ and $k\geq 75,$ $|Z''|\leq |N_{G_{k-2}'}(Z'')|$. If $H(Z)$ has no vertex that is not adjacent to the edges of $H(Z''),$ then the edges of $R$ join vertices of $H(Z'')$ with each other.
We have that $|Z''|\leq k(k-1)/2-k=k(k-3)/2$ and $|Z''|\leq |N_{G_{k-2}'}(Z'')|$.

Suppose that $H(Z'')$ has at least $k+1$ vertices. Then $R$ has at least $k+1$ edges and $|Z'|\leq |Z|-(k+1)\leq k(k-3)/2$.
As $|N_{G_{k-2}'}(Z'')|\geq (k+1)(k-3)/2,$  $|Z''|\leq |N_{G_{k-2}'}(Z'')|$.

\medskip
\noindent{\bf Case c).} $q\geq k+1$. Then $H(Z'')$ has at least $k+1$ vertices. We have that $R$ has at least $k+1$ edges and $|Z'|\leq |Z|-(k+1)\leq k(k-3)/2$. Because $|N_{G_{k-2}'}(Z'')|\geq (k+1)(k-3)/2,$  $|Z''|\leq |N_{G_{k-2}'}(Z'')|$.

\medskip
We conclude that  for any $Z''\subseteq Z',$ $|Z''|\leq |N_{G_{k-2}'}(Z'')|$. Hence, 
$G_{k-2}'$ has a matching $M$ of size $r$ such that every vertex of $Z'$ is saturated in $M$.

Clearly, $M$ is a matching in $G_k'$ as well. Recall that $R=x_1\ldots x_px_1$ is a Hamiltonian cycle in $H(Z)$
and $U=\{u_{x_1x_2},\ldots,u_{x_{p-1}x_1}\}$. For $i\in\{1,\ldots,p\},$ let $t_i$ be the number of vertices in $\{v_{x_i}^1,\ldots,v_{x_i}^s\}$ that are saturated in $M$. Because $M$ is a matching in $G_{k-1}',$ $t_i\leq s-1$. 

We prove that there is $j\in\{1,\ldots,p\}$ such that $t_j<s-1$. Let $q$ be the smallest positive integer such that $|Z|\leq q(q-1)/2$.
The graph $H(Z)$ has at least $q$ vertices. Suppose first that it has exactly $q$ vertices. Then $p=q$ and $Z'=Z\setminus U$ has at most $p(p-1)/2-p=p(p-3)/2$ vertices. Also $|N_{G_{k-2}'}(Z)|=p(k-3)/2$. If $p<k,$ at least one vertex in $N_{G_{k-2}'}(Z)$ is not saturated and the statement holds. Let $p=k$. Then because $G$ has no cliques of size $k,$ $|Z|<k(k-1)/2$ and $|Z'|<k(k-3)/2$. We have that $|Z'|<|N_{G_{k-2}'}(Z)|$ and at least one vertex in $N_{G_{k-2}'}(Z)$ is not saturated. If $p\geq k+1,$ then 
$|Z|=k(k-1)/2+1$. We have that $|Z'|\leq k(k-3)/2$ and $|N_{G_{k-2}'}(Z)|\geq (k+1)(k-3)/2$. Hence, the there is a non-saturated vertex in  $N_{G_{k-2}'}(Z)$.
Suppose now that $H(Z)$ has at least $q+1$ vertices. Then $p\geq q+1$ and $|Z'|\leq q(q-1)/2-(q+1)=q(q-3)/2-1$. As  $|N_{G_{k-2}'}(Z)|\geq (q+1)(k-3)/2,$ $|Z'|<|N_{G_{k-2}'}(Z)|$ if $q\leq k$. If $q\geq k+1,$ then $|Z'|\leq |Z|-(k+2)\leq
(k(k-1)/2+1) - (k+2)\leq k(k-3)/2-1$. Because $|N_{G_{k-2}'}(Z)|\geq (k+2)(k-3)/2,$ we again have 
a  non-saturated vertex in  $N_{G_{k-2}'}(Z)$. We considered all cases and conclude that at least one vertex of 
$N_{G_{k-2}'}(Z)$ is not saturated in $M$. Hence, there is $j\in\{1,\ldots,p\}$ such that $t_j<s-1$. Without loss of generality we assume that $j=p$. 

Because $v_{x_i}^1,\ldots,v_{x_i}^s$ have the same neighborhoods, we assume without loss of generality that for  $i\in\{1,\ldots,p\},$ $v_{x_i}^1,\ldots,v_{x_i}^{t_i}$ are saturated.  
For $i\in\{1,\ldots,q\}$ and $j\in\{1,\ldots,t_i\},$ denote by $u_i^j$ the vertex of $Z'$ such that $\{v_{x_i}^j,u_i^j\}\in M$. 
Notice that it can happen that $t_i=0$ and we have no such saturated vertices.
We define the path $P_i=v_{x_i}^1u_i^1v_{x_i}^2\ldots u_i^{s_i}v_{x_i}^{t_i+1}$ if $t_i\geq 1$ and let $P_i=v_{x_i}^1$ if $t_i=0$
for $i\in\{1,\ldots,p\}$. Let $P=P_1+v_{x_1}^{t_1+1}u_{x_1x_2}v_{x_2}^{1}+\ldots+v_{x_{p-1}}^{t_{p-1}+1}u_{x_{p-1}x_p}v_{x_p}^{1} + P_p$ and then form the cycle $C$ from $P$ by joining the end-vertices of $P$ by $v_{x_{p}}^{t_p+1} v_{x_{p}}^{t_p+2} u_{x_{p}x_1}v_{x_1}^{1}$
using the fact that $t_p\leq s-2$.
We have that  $Z\subseteq V(C)\subseteq Z\cup N_G(Z)$ and $v_{x_{p}}^{t_p+1} v_{x_{p}}^{t_p+2}\in E(C)$. It concludes Case 3 and the proof of the claim. 
\end{proof}

Let $Z\subseteq V(G'_k)$  be a set of size $k'$. Let $Z'=Z\cap\{u_{xy}|\{x,y\}\in E(G)\}$. If $Z'=\emptyset,$ then $Z$ is a clique and 
there is a cycle $C$ in $G'_k$ such that $Z\subseteq V(C)$. Suppose that $Z'\neq\emptyset$. By Claim, 
there is a cycle $C'$ in $G'_k$ such that $Z'\subseteq V(C')\subseteq Z'\cup N_G(Z')$ and
$C'$ has an edge $\{v_x^i,v_y^j\}$ for some $x,y\in V(G')$ and $i,j\in\{1,\ldots,s\}$. Let $\{u_1,\ldots,u_p\}=Z\setminus V(C')$. Notice that these vertices are pairwise adjacent and adjacent to $v_x^i,v_y^j$. We construct the cycle $C$ from $C'$ by replacing 
 $\{v_x^i,v_y^j\}$ by the path  $v_x^iu_1\ldots u_pv_y^j$. It remains to observe that $Z\subseteq V(C)\subseteq Z\cup N_G(Z)$.
\end{proof}

\subsection{Kernelization lower bounds for planar graphs}
Now we show that it is unlikely that \textsc{Cyclability}, parameterized by $k$, has a polynomial kernel when restricted to planar graphs. 
The proof uses the cross-composition technique introduced by Bodlaender, Jansen, and Kratsch~in~\cite{BodlaenderJK14}.

Let $L\subseteq\Sigma^*\times\mathbb{N}$ be a parametrized problem. Recall that a kernelization for a parameterized problem $L$ is an algorithm that takes an instance $(x,k)$ and maps it in time polynomial in $|x|$ and $k$ to an instance $(x',k')$ such that
\begin{itemize}
\item[i)] $(x,k)\in L$ if and only if $(x',k')\in L$,
\item[ii)] $|x'|$ is bounded by a computable function $f$ in $k$, and
\item[iii)] $k'$ is bounded by a computable function $g$ in $k$.
\end{itemize}
The output $(x',k')$ of kernelization is a \emph{kernel} and the function $f$ is the size of the kernel.  A kernel is \emph{polynomial} if $f$ is polynomial.

We also need the following additional definitions (see~\cite{BodlaenderJK14}). Let $\Sigma$ be a finite alphabet. An equivalence relation $\mathcal{R}$ on the set of strings $\Sigma^*$ is called a \emph{polynomial equivalence relation} if the following two conditions hold:
\begin{itemize}
\item[i)] there is an algorithm that given two strings $x,y\in\Sigma^*$ decides whether $x$ and $y$ belong to
the same equivalence class in time polynomial in $|x|+|y|$,
\item[ii)] for any finite set $S\subseteq\Sigma^*$, the equivalence relation $\mathcal{R}$ partitions the elements of $S$ into a
number of classes that is polynomially bounded in the size of the largest element of $S$.
\end{itemize}

Let $L\subseteq\Sigma^*$ be a language, let $\mathcal{R}$ be a polynomial
equivalence relation on $\Sigma^*$, and let $\mathcal{Q}\subseteq\Sigma^*\times\mathbb{N}$   
be a parameterized problem.  An \emph{AND-cross-composition of $L$ into $\mathcal{Q}$} (with respect to $\mathcal{R}$) is an algorithm that, given $t$ instances $x_1,x_2,\ldots,x_t\in\Sigma^*$ 
of $L$ belonging to the same equivalence class of $\mathcal{R}$, takes time polynomial in
$\sum_{i=1}^t|x_i|$ and outputs an instance $(y,k)\in \Sigma^*\times \mathbb{N}$ such that:
\begin{itemize}
\item[i)] the parameter value $k$ is polynomially bounded in $\max\{|x_1|,\ldots,|x_t|\} + \log t$,
\item[ii)] the instance $(y,k)$ is a \yes-instance for $\mathcal{Q}$ if and only each instance $x_i$ is a \yes-instance for $L$ for $i\in\{1,\ldots,t\}$.
\end{itemize}
It is said that $L$ \emph{AND-cross-composes into} $\mathcal{Q}$ if a cross-composition
algorithm exists for a suitable relation $\mathcal{R}$.

In particular, Bodlaender, Jansen and Kratsch~\cite{BodlaenderJK14} proved the following theorem.

\begin{theorem}[\cite{BodlaenderJK14}]\label{thm:BJK}
If an {\sf NP}-hard language $L$ AND-cross-composes into the parameterized problem $\mathcal{Q}$,
then $\mathcal{Q}$ does not admit a polynomial kernelization unless
${\sf NP}\subseteq{\sf co}\mbox{-}{\sf NP}/{\sf poly}$.
\end{theorem}

We consider the auxiliary \textsc{Hamiltonicity with a Given Edge} problem that for a graph $G$ with a given edge $e$, asks whether $G$ has a Hamiltonian cycle that contains $e$.  
We use the following lemma.

\begin{lemma}\label{lem:Ham-edge}
\textsc{Hamiltonicity with a Given Edge} is  {\sf NP}-complete for cubic planar graphs.
\end{lemma}

\begin{proof}
It was proved by Garey, Johnson and Tarjan in~\cite{GareyJT76} that {\sc Hamiltonicity}  is  {\sf NP}-complete for planar cubic graphs. Let
$G$ be a planar cubic graph, and let $v$ be an arbitrary vertex of $G$. Denote by $x,y,z$ the neighbors of $v$ in $G$. We replace $v$ by a gadget $F$ shown in Fig.~\ref{fig:F}. More precisely, we delete $v$, construct a copy of $F$ and add edges $\{x,x\rq{}\}$, $\{y,y\rq{}\}$ and $\{z,z\rq{}\}$. Denote by $G\rq{}$ the obtained graph. Clearly, $G\rq{}$ is a cubic planar graph.
We claim that $G$ is Hamiltonian if and only if $G\rq{}$ has a Hamiltonian cycle that contains the edge $e$ shown in Fig.~\ref{fig:F}.

\begin{figure}[ht]
\begin{center}
    \includegraphics[width=1\textwidth]{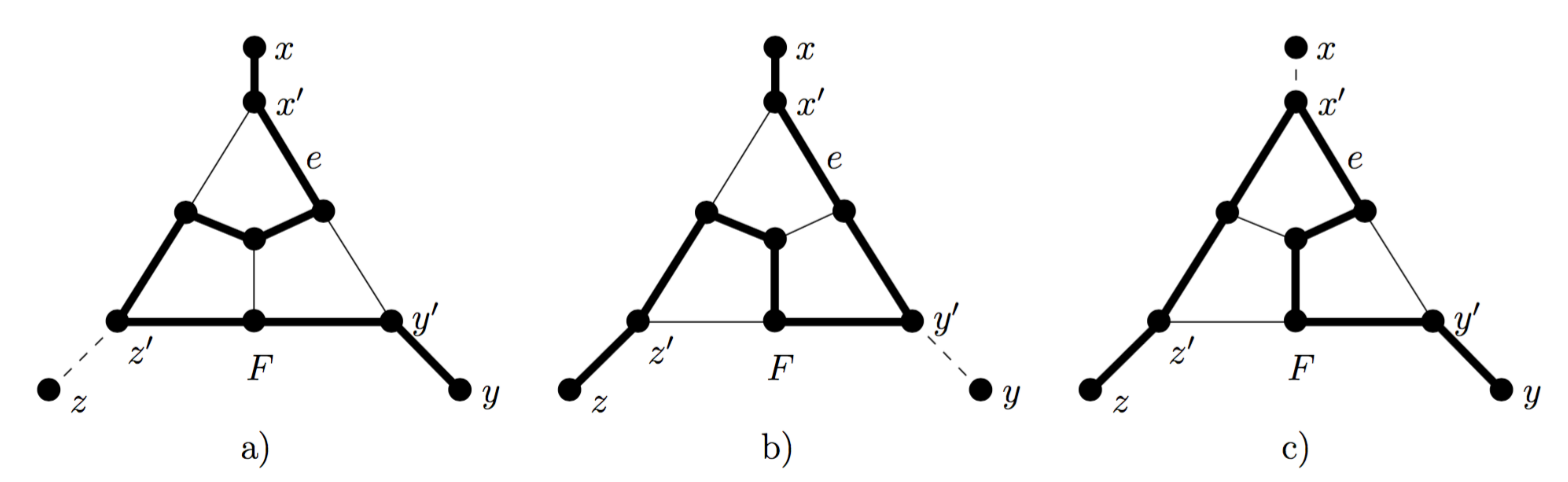}
\end{center}
\caption{The gadget $F$; the edges of Hamiltonian cycles are shown by the bold lines. 
\label{fig:F}}
\end{figure}

Suppose that $G$ has a Hamiltonian cycle $C$. Then $C$ contains two edges incident to $v$. We construct the Hamiltonian cycle in $G\rq{}$ by replacing these two edges by paths shown in Fig.~\ref{fig:F}. If $C$ contains $\{x,v\}$ and $\{v,y\}$, then they are replaced by the path shown in Fig.~\ref{fig:F} a), if $C$ contains $\{x,v\}$ and $\{v,z\}$, then they are replaced by the path shown in Fig.~\ref{fig:F} b) and if $C$ contains $\{y,v\}$ and $\{v,z\}$, then we use the path shown in Fig.~\ref{fig:F} c). It is easy to see that we obtain a Hamiltonian cycle that contains $e$. If $G\rq{}$ has a Hamiltonian cycle, then it is straightforward to see that $G$ is Hamiltonian as well.
\end{proof}

Now we are ready to prove  Theorem~\ref{thm:no-kernel}.


\begin{proof}[of Theorem~\ref{thm:no-kernel}]
We construct an AND-cross-composition of \textsc{Hamiltonicity with a Given Edge}. By Lemma~\ref{lem:Ham-edge}, the problem is {\sf NP}-complete.
We assume that  two instances $(G,e)$ and $(G',e\rq{})$ of \textsc{Hamiltonicity with a Given Edge}  are equivalent if $|V(G)|=|V(G')|$. 
Let $(G_i,e_i)$ for $i\in\{1,\ldots,t\}$ be  equivalent instances of \textsc{Hamiltonicity with a Given Edge}, $|V(G_i)|=n$. 
We construct the graph $G$ as follows (see Fig.~\ref{fig:G}).  
\begin{itemize}
\item[i)] Construct disjoint copies of $G_1,\ldots,G_t$.
\item[ii)] For each $i\in\{1,\ldots,t\}$, subdivide $e_i$ twice and denote the obtained vertices by $u_i,v_i$.
\item[iii)] For $i\in \{1,\ldots,t\}$, construct an edge $\{v_i,u_{i+1}\}$ assuming that $u_{n+1}=u_1$.
\end{itemize}
It is straightforward to see that $G$ is a cubic planar graph.

\begin{figure}[ht]
\begin{center}
    \includegraphics[width=0.65\textwidth]{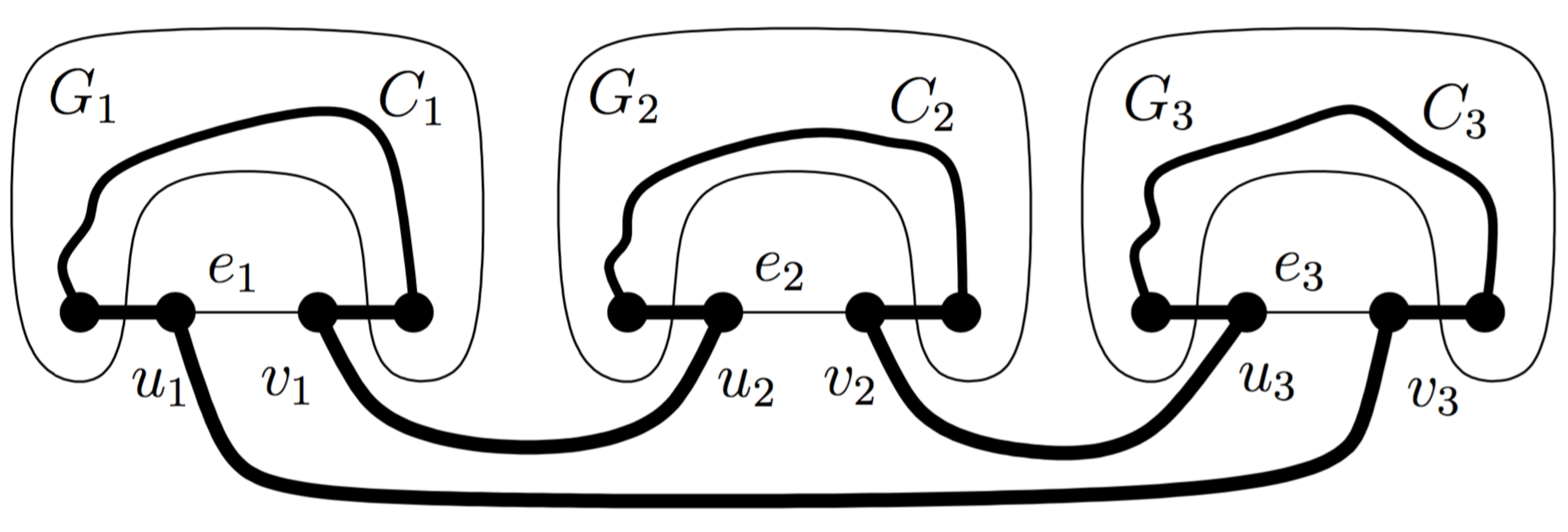}
\end{center}
\caption{The construction of $G$ for $t=3$; the edges of a Hamiltonian cycle in $G$ are shown by the bold lines. 
\label{fig:G}}
\end{figure}

We claim that $G$ is $n+2$-cyclable if and only if $(G_i,e_i)$ is a \yes-instance of \textsc{Hamiltonicity with a Given Edge} for every $i\in\{1,\ldots,t\}$. If every $G_i$ has a Hamiltonian cycle $C_i$ that contains $e_i$, then $G$ is Hamiltonian as well; the Hamiltonian cycle in $G$ is constructed from $C_1,\ldots,C_t$ as it is shown in Fig.~\ref{fig:G}.  Since $G$ is Hamiltonian, $G$ is $n+2$-cyclable. Suppose now that $G$ is $n+2$-cyclable. Let $i\in\{1,\ldots,t\}$. Consider $X=V(G_i)\cup \{u_i,v_i\}$. Because $|X|=n+2$, $G$ has a cycle $C$ that goes trough all the vertices of $X$.
 It remains to observe that by the removal of the vertices of $V(G)\setminus V(G_i)$ and by the addition of the edge $e_i$, we obtain from $C$ a Hamiltonian cycle in $G_i$ that contains $e_i$. 
\end{proof}

\section{Discussion}
\label{discisert}

Notice that we have no proof (or evidence) that {\sc Cyclability} is in {\sc NP}. The definition of the problem classifies it directly in ${\sf \Pi}_{2}^{\rm P}$. This prompts us to conjecture the following: 

\begin{conjecture}
\label{adeds}
{\sc Cyclability}  is  ${\sf \Pi}_{2}^{\rm P}$-complete.
\end{conjecture}

Moreover, while we have proved that {\sc Cyclability}  is  {\sf co-W[1]}-hard, we have no evidence which level of the parameterized complexity hierarchy it belongs to (lower than the {\sf XP} class). We find it an intriguing question whether there is some $i\geq 1$
for which {\sc Cyclability}  is  ${\sf W}[i]$-complete (or ${\sf co}\mbox{-}{\sf W}[i]$-complete).

Clearly, a challenging question  is whether  the, double exponential, parametric dependance of our {\sf FPT}-algorithm can be improved.
We believe that this is not possible and we suspect that this issue might be related to Conjecture~\ref{adeds}.

Another direction of research is whether {\sc Cyclability}  is in  {\sf FPT} on more general graph classes.
Actually, all results that were used for our algorithm can be extended on graphs embeddable on surfaces 
of bounded genus -- see~\cite{GeelenRS04,DemaineHT06theb,RobertsonS-GMXIII,RobertsonS-XXI,Kawarabayashi:W10asho} -- and yield an {\sf FPT}-algorithm on such graphs  (with worst time bounds). 
We believe that this is still the case 
for graph classes excluding some fixed graph as a minor. However, in our opinion, such an extension even though possible would be too technically involved.



\end{document}